\documentclass[11pt]{article}
\def\thetitle{On the trace of random walks on random graphs}

\usepackage{amsmath,amssymb}

\usepackage[usenames]{xcolor}
\definecolor{CombinatoricaAqua}{HTML}{00698C}
\definecolor{CombinatoricaBlue}{HTML}{3A3293}
\definecolor{CombinatoricaBrown}{HTML}{66220C}
\definecolor{CombinatoricaRed}{HTML}{DF2A27}
\definecolor{HarvardCrimson}{rgb}{0.6471, 0.1098, 0.1882}

\usepackage[backref=page]{hyperref}
\hypersetup{
  unicode,
  pdfencoding=auto,
  pdfauthor={Alan Frieze, Michael Krivelevich, Peleg Michaeli, Ron Peled},
  pdftitle={\thetitle},
  pdfsubject={05C81 (Primary) 05C45, 05C80, 60G50 (Secondary)},
  pdfkeywords={random walk, random graph, Hamilton cycle},
  colorlinks=true,
  citecolor=CombinatoricaBlue,
  linkcolor=CombinatoricaAqua,
  anchorcolor=CombinatoricaBrown,
  urlcolor=HarvardCrimson}

\usepackage{amsthm}
\usepackage{thmtools}

\usepackage[capitalize]{cleveref}


\declaretheorem[name=Theorem]{ourthm}
\declaretheorem[name=Corollary,sibling=ourthm]{ourcor}
\declaretheorem[parent=section]{theorem}
\declaretheorem[sibling=theorem]{lemma}
\declaretheorem[sibling=theorem]{corollary}
\declaretheorem[sibling=theorem]{claim}

\declaretheorem[sibling=theorem]{definition}

\declaretheorem[style=definition,numbered=no]{acknowledgement}

\usepackage[nobysame,msc-links]{amsrefs}

\BibSpec{book}{%
    +{}  {\PrintPrimary}                {transition}
    +{,} { \textbf}                     {title} 
    +{.} { }                            {part}
    +{:} { \textit}                     {subtitle}
    +{,} { \PrintEdition}               {edition}
    +{}  { \PrintEditorsB}              {editor}
    +{,} { \PrintTranslatorsC}          {translator}
    +{,} { \PrintContributions}         {contribution}
    +{,} { }                            {series}
    +{,} { \voltext}                    {volume}
    +{,} { }                            {publisher}
    +{,} { }                            {organization}
    +{,} { }                            {address}
    +{,} { \PrintDateB}                 {date}
    +{,} { }                            {status}
    +{}  { \parenthesize}               {language}
    +{}  { \PrintTranslation}           {translation}
    +{;} { \PrintReprint}               {reprint}
    +{.} { }                            {note}
    +{.} {}                             {transition}
    +{}  {\SentenceSpace \PrintReviews} {review}
}

\makeatletter
\renewcommand{\PrintNames@a}[4]{%
  \PrintSeries{\name}
    {#1}
    {}{ and \set@othername}
    {,}{ \set@othername}
    {}{ and \set@othername}
    {#2}{#4}{#3}%
}
\makeatother

\usepackage{soul}
\soulregister\ref7

\usepackage[textsize=tiny,backgroundcolor=black!10]{todonotes}
\presetkeys{todonotes}{noline}{}

\usepackage[a4paper]{geometry}
\title{\thetitle}
\author{Alan Frieze
  \thanks{Department of Mathematical Sciences, Carnegie Mellon University,
    Pittsburgh PA15213, USA.  Email: {\tt alan@random.math.cmu.edu}.
    Research supported in part by NSF Grant ccf0502793.}
\and Michael Krivelevich
\thanks{School of Mathematical Sciences, Raymond and Beverly Sackler Faculty
  of Exact Sciences, Tel Aviv University, Tel Aviv 6997801, Israel. E-mail:
  {\tt krivelev@post.tau.ac.il}. Research supported in part by a USA-Israel
BSF Grant and by a grant from Israel Science Foundation.}
\and Peleg Michaeli
\thanks{School of Mathematical Sciences, Raymond and Beverly Sackler Faculty
  of Exact Sciences, Tel Aviv University, Tel Aviv 6997801, Israel. E-mail:
  {\tt peleg.michaeli@math.tau.ac.il}.}
\and Ron Peled
\thanks{School of Mathematical Sciences, Raymond and Beverly Sackler Faculty
  of Exact Sciences, Tel Aviv University, Tel Aviv 6997801, Israel. E-mail:
  {\tt peledron@post.tau.ac.il}.  Research supported in part by an IRG grant
and by a grant from Israel Science Foundation.}}

\usepackage{mleftright}
\mleftright

\newcommand{\lln}[1]{\ln{\ln{#1}}}
\newcommand{\floor}[1]{\left\lfloor{#1}\right\rfloor}
\newcommand{\ceil}[1]{\left\lceil{#1}\right\rceil}

\newcommand{\gnm}[0]{G(n,m)}
\newcommand{\gnp}[0]{G(n,p)}
\newcommand{\Small}[0]{\mathrm{SMALL}}

\DeclareMathOperator{\bin}{Bin}
\DeclareMathOperator{\hypg}{Hypergeometric}
\DeclareMathOperator{\unif}{Unif}

\newcommand{\whp}[0]{\textbf{whp}}
\newcommand{\dtv}[2]{d_{\mathrm{TV}}(#1,#2)}

\newcommand{\bollobas}[0]{{Bollob\'as}}
\newcommand{\cerny}[0]{{{\v{C}}ern{\'y}}}
\newcommand{\erdos}[0]{{Erd\H{o}s}}
\newcommand{\lovasz}[0]{{Lov{\'a}sz}}
\newcommand{\komlos}[0]{{Koml\'os}}
\newcommand{\renyi}[0]{{R\'enyi}}
\newcommand{\szemeredi}[0]{{Szemer\'edi}}

\newcommand{\midd}[0]{\ \big\vert\ }

\newcommand{\odd}[1]{#1^{\mathrm{o}}}
\newcommand{\even}[1]{#1^{\mathrm{e}}}

\makeatletter
\newcommand*{\E}{%
  \def\E@sub{}%
  \def\E@sup{}%
  \E@scripts
}
\newcommand*{\E@scripts}{%
  \@ifnextchar_\E@subscript{%
    \@ifnextchar^\E@supscript\E@finish
  }%
}
\def\E@subscript_#1{%
  \ifx\E@sub\@empty
    \def\E@sub{#1}%
  \else
    \errmessage{E already has a subscript}%
  \fi
  \E@scripts
}
\def\E@supscript^#1{%
  \ifx\E@sup\@empty
    \def\E@sup{#1}%
  \else
    \errmessage{E already has a superscript}%
  \fi
  \E@scripts
}
\newcommand*{\E@finish}[1]{%
  \mathbb{E}%
  \ifx\E@sub\@empty\else _{\E@sub}\fi
  \ifx\E@sup\@empty\else ^{\E@sup}\fi
  \left(#1\right)%
}

\newcommand*{\pr}{%
  \def\pr@sub{}%
  \def\pr@sup{}%
  \pr@scripts
}
\newcommand*{\pr@scripts}{%
  \@ifnextchar_\pr@subscript{%
    \@ifnextchar^\pr@supscript\pr@finish
  }%
}
\def\pr@subscript_#1{%
  \ifx\pr@sub\@empty
    \def\pr@sub{#1}%
  \else
    \errmessage{E already has a subscript}%
  \fi
  \pr@scripts
}
\def\pr@supscript^#1{%
  \ifx\pr@sup\@empty
    \def\pr@sup{#1}%
  \else
    \errmessage{E already has a superscript}%
  \fi
  \pr@scripts
}
\newcommand*{\pr@finish}[1]{%
  \mathbb{P}%
  \ifx\pr@sub\@empty\else _{\pr@sub}\fi
  \ifx\pr@sup\@empty\else ^{\pr@sup}\fi
  \left(#1\right)%
}
\makeatother

\makeatletter
\def\namedlabel#1#2{\begingroup
    #2%
    \def\@currentlabel{#2}%
    \phantomsection\label{#1}\endgroup
}
\makeatother

\begin{document}
\maketitle

\begin{abstract}
  We study graph-theoretic properties of the trace of a random walk on a
  random graph. We show that for any $\varepsilon>0$ there exists $C>1$ such
  that the trace of the simple random walk of length $(1+\varepsilon)n\ln{n}$
  on the random graph $G\sim\gnp$ for $p>C\ln{n}/n$ is, with high
  probability, Hamiltonian and $\Theta(\ln{n})$-connected.  In the special
  case $p=1$ (i.e.\ when $G=K_n$), we show a hitting time result according to
  which, with high probability, exactly one step after the last vertex has
  been visited, the trace becomes Hamiltonian, and one step after the last
  vertex has been visited for the $k$'th time, the trace becomes
  $2k$-connected.
\end{abstract}

\section{Introduction}
Since the seminal study of \erdos{} and \renyi{}~\cite{ER_evo}, random graphs
have become an important branch of modern combinatorics.  It is an interesting
and natural concept to study for its own sake, but it has also proven to have
numerous applications both in combinatorics and in computer science.  Indeed,
random graphs have been a subject of intensive study during the last 50 years:
thousands of papers and at least three books~\cites{JLR,B01,FK} are devoted
to the subject.  The term \emph{random graph}  is used to refer to several
quite different models, each of which is essentially a distribution over
all graphs on $n$ labelled vertices.  Perhaps the two most famous models are
the classical models $\gnm$, obtained by choosing $m$ edges uniformly at
random among the $\binom{n}{2}$ possible edges, and $\gnp$, obtained by
selecting each edge independently with probability $p$.  Other models are
discussed in~\cite{FK}.

In this paper, we study a different model of random graphs -- the (random)
graph formed by the trace of a simple random walk on a finite graph.  Given a
base graph and a starting vertex, we select a vertex uniformly at random from
its neighbours and move to this neighbour, then independently select a vertex
uniformly at random from that vertex's neighbours and move to it, and so on.
The sequence of vertices this process yields is a \emph{simple random walk}
on that graph.  The set of vertices in this sequence is called the
\emph{range} of the walk, and the set of edges traversed by this walk is
called the \emph{trace} of the walk.  The literature on the topic of random
walks is vast; however, most effort was put into answering questions about
the range of the walk, or about the distribution of the position of the walk
at a fixed time.  Examples include estimating the \emph{cover time} (the time
it takes a walk to visit all vertices of the graph) and the \emph{mixing
time} of graphs (see \lovasz{}~\cite{LOV96} for a survey).  On the other
hand, to the best of our knowledge, there are almost no works addressing
explicitly questions about the trace of the walk.  One paper of this type is
that of Barber and Long~\cite{BL13}, discussing the quasirandomness of the
trace of long random walks on dense quasirandom graphs.  We mention here that
there are several papers studying the subgraph induced by vertices that are
\emph{not} visited by the random walk (see for example Cooper and
Frieze~\cite{CF13}, and \cerny{}, Teixeira and Windisch~\cite{CTW11}).  We also
mention that on infinite graphs, several properties of the trace have been
studied (for an example see~\cite{BGL07}).

Our study focuses on the case where the base graph $G$ is random and
distributed as $\gnp$.  We consider the graph $\Gamma$ on the same vertex set
($[n]=\left\{1,\ldots,n\right\}$), whose edges are the edges traversed by the
random walk on $G$.  A natural graph-theoretic question about $\Gamma$ is
whether it is connected.  A basic requirement for that to happen is that the
base graph is itself connected.  It is a well-known result (see~\cite{ER_rgI})
that in order to guarantee that $G$ is connected, we must take
$p>(\ln{n}+\omega(1))/n$.  Given that our base graph is indeed connected, for
the trace to be connected, the walk must visit all vertices.  An important
result by Feige~\cite{F95} states that for connected graphs on $n$ vertices,
this happens on average after at least $(1-o(1))n\ln{n}$ steps.  Cooper and
Frieze~\cite{CF07_srg} later gave a precise estimation for the average cover
time of (connected) random graphs, directly related to how large $p$ is, in
comparison to the connectivity threshold.  In fact, it can be derived from
their proof that if $p=\Theta(\ln{n}/n)$ and the length of the walk is at
most $n\ln{n}$, then the trace is typically not connected.

It is thus natural to execute a random walk of length
$(1+\varepsilon)n\ln{n}$ on a random graph which is above the connectivity
threshold by at least a large constant factor (which may depend on
$\varepsilon$), and to ask what other graph-theoretic properties the trace
has.  For example, is it highly connected?  Is it Hamiltonian?  The set of
visited vertices does not reveal much information about the global structure
of the graph, so the challenge here is to gain an understanding of that
structure by keeping track of the traversed edges.  What we essentially show
is that the trace is typically Hamiltonian and
$\Theta(\ln{n})$-vertex-connected.  Our method of proof will be to show that
the set of traversed edges typically forms an expander.

In the boundary case where $p=1$, i.e.\ when the base graph is $K_n$, we
prove a much more precise result.  As the trace becomes connected exactly
when the last vertex has been visited, and at least one more step is required
for that last visited vertex to have degree $2$ in the trace, one cannot hope
that the trace would contain a Hamilton cycle beforehand.  It is reasonable
to expect however that this degree requirement is in fact the bottleneck for
a typical trace to be Hamiltonian, as is the case in other random graph
models.  In this paper, we show a hitting time result according to which,
with high probability, one step after the walk connects the subgraph (that
is, one time step after the cover time), the subgraph becomes Hamiltonian.
This result implies that the bottleneck to Hamiltonicity of the trace lies
indeed in the minimum degree, and in that sense, the result is similar in
spirit to the results of \bollobas{}~\cite{Bol84}, and of Ajtai, \komlos{}
and \szemeredi{}~\cite{AKS85}.  We also extend this result for $k$-cover-time
vs.\ minimum degree $2k$ vs.\ $2k$-vertex-connectivity, obtaining a result
similar in spirit to the result of \bollobas{} and Thomason~\cite{BT85}.

Let us try to put our result in an appropriate context. The accumulated
research experience in several models of random graphs allows to predict
rather naturally that in our model (of the trace of a random walk on a random
graph, or on the complete graph) the essential bottleneck for the
Hamiltonicity would be existence of vertices of degree less than two. It is
not hard to prove that typically the next step after entering the last
unvisited vertex creates a graph of minimum degree two, and thus the
appearance of a Hamilton cycle should be tightly bundled with the cover time
of the random walk. Thus our results, confirming the above paradigm, are
perhaps expected and not entirely surprising. Recall however that it took a
fairly long and quite intense effort of the random graphs community, to go
from the very easy, in modern terms at least, threshold for minimum degree
two to the much more sophisticated result on Hamiltonicity, still considered
one of the pinnacles of research in the subject. Here too we face similar
difficulties; our proof, while having certain similarities to the prior
proofs of Hamiltonicity in random graphs, still has its twists and turns to
achieve the goal. In particular, for the case of a random walk on a random
graph, we cast our argument in a pseudo-random setting and invoke a
sufficient deterministic criterion for Hamiltonicity due to Hefetz,
Krivelevich and Szab\'o \cite{HKS09}. The hitting time result about a random
walk on the complete graph is still a delicate argument, as typical for
hitting time results, due to the necessity to pinpoint the exact point of the
appearance of a Hamilton cycle. There we look at the initial odd steps of the
walk, forming  a random multigraph with a given number of edges, and then add
just the right amount of some further odd and even steps to hit the target
exactly when needed.

\subsection{Notation and terminology}
Let $G$ be a (multi)graph on the vertex set $[n]$. For two vertex sets
$A,B\subseteq[n]$, we let $E_G(A,B)$ be the set of edges having one endpoint
in $A$ and the other in $B$. If $v\in[n]$ is a vertex, we may write
$E_G(v,B)$ when we mean $E_G(\left\{v\right\},B)$. We denote by $N_G(A)$ the
\emph{external} neighbourhood of $A$, i.e., the set of all vertices in
$[n]\smallsetminus A$ that have a neighbour in $A$. Again, we may write
$N_G(v)$ when we mean $N_G(\left\{v\right\})$. We also write
$N^+_G(A)=N_G(A)\cup A$. The degree of a vertex $v\in[n]$, denoted by
$d_G(v)$, is its number of incident edges, where a loop at $u$ contributes $2$
to its degree. The \emph{simplified graph} of $G$ is the simple graph $G'$
obtained by replacing each multiedge with a single edge having the same
endpoints, and removing all loops. The \emph{simple degree} of a vertex is
its degree in the simplified graph; it is denoted by $d'_G(v)=d_{G'}(v)$. We
let $\delta(G)$ and $\Delta(G)$ be the minimum and maximum \emph{simple}
degrees ($d'$) of $G$.  Let the \emph{edge boundary} of a vertex set $S$ be
the set of edges of $G$ with exactly one endpoint in $S$, and denote it by
$\partial_G{S}$ (thus $\partial_G(S)$ is the set of edge of the cut
$(S,S^c)$ in $G$).  If $v,u$ are distinct vertices of a graph $G$, the
\emph{distance from $v$ to $u$} is defined to be the minimum length (measured
in edges) of a path from $v$ to $u$ (or $\infty$ if there is no such path); it
is denoted by $d_G(v, u)$.  If $v$ is a vertex, the \emph{ball of radius $r$
around $v$} is the set of vertices of distance at most $r$ from $v$; it is
denoted by $B_G(v,r)$. In symbols:
\begin{equation*}
  B_G(v,r) = \left\{u\in [n]\mid d_G(v, u)\le r\right\}.
\end{equation*}
We also write $N_G(v,r)=B_G(v,r)\smallsetminus B_G(v,r-1)$ for the (inner)
\emph{vertex boundary} of that ball.  We will sometimes omit the subscript $G$
in the above notations if the graph $G$ is clear from the context.

A \emph{simple random walk} of length $t$ on a graph $G$, starting at a vertex
$v$, is denoted $(X_i^{v}(G))_{i=0}^t$. When the graph is clear from the
context, we may omit it and simply write $(X_i^{v})_{i=0}^t$.  When the
starting vertex is irrelevant, we may omit it as well, writing $(X_i)_{i=0}^t$.
In \cref{sec:preliminaries,sec:pseudo,sec:hittingtime} we shall define ``lazy''
versions of a simple random walk, for which we will use the same notation.  The
\emph{trace} of a simple random walk on a graph $G$ of length $t$, starting at
a vertex $v$, is the subgraph of $G$ having the same vertex set as $G$, whose
edges are all edges traversed by the walk (including loops), counted with
multiplicity (so it is in general a multigraph). We denote it by
$\Gamma_t^{v}(G)$, $\Gamma_t^{v}$ or $\Gamma_t$, depending on the context.

For a positive integer $n$ and a real $p\in[0,1]$, we denote by $\gnp$ the
probability space of all (simple) labelled graphs on the vertex set $[n]$
where the probability of each such a graph, $G=([n],E)$, to be chosen is
$p^{|E|}(1-p)^{\binom{n}{2}-|E|}$.
If $f,g$ are functions of $n$ we use the notation $f\sim g$ to denote
asymptotic equality.  That is, $f\sim g$ if and only if
$\lim_{n\to\infty} f/g=1$.

For the sake of simplicity and clarity of presentation, we do not make a
particular effort to optimize the constants obtained in our proofs.
We also omit floor and ceiling signs whenever these are not crucial. Most of
our results are asymptotic in nature and whenever necessary we assume that $n$
is sufficiently large.  We say that an event holds \emph{with high
probability} (\whp{}) if its probability tends to $1$ as $n$ tends to
infinity.

\subsection{Our results}
Our first theorem states that if $G\sim\gnp{}$ with $p$ above the
connectivity threshold by at least some large constant factor, and the walk
on $G$ is long enough to traverse the expected number of edges required to
make a random graph connected, then its trace is with high probability
Hamiltonian and highly connected.

\begin{ourthm}\label{thm:gnp}
  For every $\varepsilon>0$ there exist $C=C(\varepsilon)>0$ and
  $\beta=\beta(\varepsilon)>0$ such that for
  every edge probability $p=p(n)\ge C\cdot \frac{\ln{n}}{n}$ and for every
  $v\in[n]$, a random graph $G\sim\gnp$ is \whp{} such that for
  $L=(1+\varepsilon)n\ln{n}$, the trace $\Gamma^v_L(G)$ of a simple random
  walk of length $L$ on $G$, starting at $v$, is \whp{} Hamiltonian and
  $(\beta\ln{n})$-vertex-connected.
\end{ourthm}

Our proof strategy will be as follows.  First we prove that \whp{}
$G\sim\gnp$ satisfies some pseudo-random properties.  Then we show that
\whp{} the trace of a simple random walk on \emph{any} given graph, which
satisfies these pseudo-random properties, has good expansion properties.
Namely, it has two properties, one ensures expansion of small sets, the
other guarantees the existence of an edge between any two disjoint large
sets.

In the next two theorems we address the case of a random walk $X$ executed on
the complete graph $K_n$, and we assume that the walk starts at a uniformly
chosen vertex.  Denote the number of visits of the random walk $X$ to a
vertex $v$ by time $t$ (including the starting vertex) by $\mu_{t}(v)$.  For
a natural number $k$, denote by $\tau_C^k$ the \emph{$k$-cover time} of the
graph by
the random walk; that is, the first time $t$ for which each vertex in $G$ has
been visited at least $k$ times. In symbols,
\begin{equation}\label{tauCk}
  \tau_C^k = \min\left\{t\mid \forall v\in [n],\ \mu_t(v)\ge k\right\}.
\end{equation}
When $k=1$ we simply write $\tau_C$ and call it the \emph{cover time} of
the graph. The objective of the following theorems is to prove that the
minimal requirements for Hamiltonicity and $k$-vertex-connectivity
are in fact the bottleneck for a typical trace to have these properties.

\begin{ourthm}\label{kn_ham}
  For a simple random walk on $K_n$, denote by $\tau_\mathcal{H}$ the hitting
  time of the property of being Hamiltonian. Then \whp{}
  $\tau_\mathcal{H}=\tau_C+1$.
\end{ourthm}

\begin{ourcor}\label{kn_pm}
  Assume $n$ is even.
  For a simple random walk on $K_n$, denote by $\tau_\mathcal{PM}$ the hitting
  time of the property of admitting a perfect matching.  Then \whp{}
  $\tau_\mathcal{PM} = \tau_C$.
\end{ourcor}

\begin{samepage}
\begin{ourthm}\label{kn_conn}
  For every $k\ge 1$, for a simple random walk on $K_n$, denote by
  $\tau_\delta^k$ the hitting time of the property of being spanning with
  minimum simple degree $k$, and denote by $\tau_\kappa^k$ the hitting time of
  the property of being spanning $k$-vertex-connected. Then \whp{}
  \begin{equation*}
    \begin{array}{rcccl}
      \tau_C^{k}    &=& \tau_\delta^{2k-1} &=& \tau_\kappa^{2k-1},\\
      \tau_C^{k} +1 &=& \tau_\delta^{2k} &=& \tau_\kappa^{2k}.
    \end{array}
  \end{equation*}
\end{ourthm}
\end{samepage}

\subsection{Organization}
The organization of the paper is as follows. In the next section we present
some auxiliary results, definitions and technical preliminaries.  In
\cref{sec:pseudo} we explore important properties of the random walk on
a pseudo-random graph.  In \cref{sec:hamcon} we prove the Hamiltonicity and
vertex-connectivity results for the trace of the walk on $\gnp$.  In
\cref{sec:hittingtime} we prove the hitting time results of the walk on $K_n$.
We end by concluding remarks and proposals for future work in
\cref{sec:concluding}.

\section{Preliminaries}\label{sec:preliminaries}
In this section we provide tools to be used by us in the succeeding sections.
We start by stating two versions of known bounds on large deviations of
random variables, due to Chernoff~\cite{Chernoff} and
Hoeffding~\cite{Hoeffding}, whose proofs can be found, e.g., in Chapter 2
of~\cite{JLR}.  Define
\begin{equation*}
\varphi(x)=\begin{cases}
  (1+x)\ln(1+x)-x & x \ge -1\\
  \infty & \text{otherwise}.
  \end{cases}
\end{equation*}

\begin{theorem}[\cite{JLR}*{Theorem 2.1}]\label{chernoff}
  Let $X\sim\bin(n,p)$, $\mu=np$, $a\ge 0$. Then the following inequalities
  hold:
  \begin{align}
    \label{chernoff_ineq_low}
    \pr{X\le \mu - a}
    &\le \exp\left(-\mu\varphi\left(\frac{-a}{\mu}\right)\right)
    \le \exp\left(-\frac{a^2}{2\mu}\right),\\
    \label{chernoff_ineq_high}
    \pr{X\ge \mu + a}
    &\le \exp\left(-\mu\varphi\left(\frac{a}{\mu}\right)\right)
    \le \exp\left(-\frac{a^2}{2(\mu+a/3)}\right).
  \end{align}
\end{theorem}

\begin{theorem}[\cite{JLR}*{Theorem 2.10}]\label{chernoff:hypg}
  Let $N\ge 0$, and let $0\le K,n\le N$ be natural numbers. Let
  $X\sim\hypg(N,K,n)$, $\mu=\E{X}=nKN^{-1}$. Then inequalities
  \eqref{chernoff_ineq_low} and \eqref{chernoff_ineq_high} hold.
\end{theorem}

Observe that for every $c>0$, letting $\ell_c(x) = -cx+1-c-e^{-c}$ we have that
$\varphi(x)\ge \ell_c(x)$ for every $x$, as $\ell_c(x)$ is the tangent line to
$\varphi(x)$ at $x=e^{-c}-1$, and $\varphi(x)$ is convex.
We thus obtain the following bound:
\begin{corollary}\label{chernoff:cor}
  Let $X\sim\bin(n,p)$, $\mu=np$, $0<\alpha<1<\beta$. Then the following
  inequalities hold for every $c>0$:
  \begin{align*}
    \pr{X\le \alpha\mu} &\le \exp\left(-\mu(1-e^{-c}-\alpha c)\right),\\
    \pr{X\ge \beta\mu} &\le \exp\left(-\mu(1-e^{-c}-\beta c)\right).
  \end{align*}
\end{corollary}

\begin{samepage}
The following is a trivial yet useful bound:
\begin{claim}\label{chernoff:trivial}
  Suppose $X\sim\bin(n,p)$. The following bound holds:
  \begin{equation*}
    \pr{X\ge k} \le \binom{n}{k}p^k \le \left(\frac{enp}{k}\right)^k.
  \end{equation*}
\end{claim}
\end{samepage}

\begin{proof}
  Think of $X$ as $X=\sum_{i=1}^n X_i$, where $X_i$ are i.i.d.\ Bernoulli
  tests with probability $p$. For any set $A\subseteq[n]$ with $|A|=k$, let
  $E_A$ be the event ``$X_i$ have succeeded for all $i\in A$''. Clearly,
  $\pr{E_A}=p^k$. If $X\ge k$, there exists $A\subseteq[n]$ for which $E_A$.
  Thus, the union bound gives
  \begin{equation*}
    \pr{X\ge k} \le\binom{n}{k}\pr{E_A} = \binom{n}{k}p^k
    \le \left(\frac{enp}{k}\right)^k.\qedhere
  \end{equation*}
\end{proof}

\subsection{\texorpdfstring{$(R,c)$}{(R,c)}-expanders}\label{sec:RCexpanders}
Let us first define the type of expanders we intend to use.
\begin{definition}
  For every $c>0$ and every positive integer $R$ we say that a graph
  $G=(V,E)$ is an \emph{$(R,c)$-expander} if every subset of vertices
  $U\subseteq V$ of cardinality $|U|\le R$ satisfies $|N_G(U)|\ge c|U|$.
\end{definition}

Next, we state some properties of $(R,c)$-expanders.

\begin{claim}\label{rc1}
  Let $G=(V,E)$ be an $(R,c)$-expander, and let $S\subseteq V$ of cardinality
  $k< c$. Denote the connected components of $G\smallsetminus S$ by
  $S_1,\ldots,S_t$, so that $1\le|S_1|\le\ldots\le|S_t|$. It follows that
  $|S_1|>R$.
\end{claim}
\begin{proof}
  Assume otherwise. Since any external neighbour of a vertex from $S_1$ must
  be in $S$, we have that
  \begin{equation*}
    c > k = |S| \ge |N(S_1)| \ge c|S_1| \ge c,
  \end{equation*}
  which is a contradiction.
\end{proof}

The following is a slight improvement of~\cite{BFHK}*{Lemma 5.1}.
\begin{lemma}\label{rc_lemma}
  For every positive integer $k$, if $G=([n],E)$ is an $(R,c)$-expander such
  that $c\ge k$ and $R(c+1)\ge\frac{1}{2}\left(n+k\right)$, then $G$ is
  $k$-vertex-connected.
\end{lemma}

\begin{proof}
  Assume otherwise; let $S\subseteq[n]$ with $|S|=k-1$ be a disconnecting set
  of vertices.  Denote the connected components of $G\smallsetminus S$ by
  $S_1,\ldots,S_t$, so that $1\le|S_1|\le\ldots\le|S_t|$ and $t\ge 2$.  It
  follows from \cref{rc1} that $|S_1|>R$.

  Take $A_i\subseteq S_i$ for $i\in[2]$ with $|A_i|=R$. Since any common
  neighbour of $A_1$ and $A_2$ must lie in $S$, it follows that
  \begin{align*}
    n &\ge |S_1\cup S_2\cup N(S_1) \cup N(S_2)|\\
      &\ge |N^+(A_1) \cup N^+(A_2)|\\
      &= |N^+(A_1)| + |N^+(A_2)| - |N(A_1)\cap N(A_2)|\\
      &\ge 2R(c+1) - |S| \ge n + 1,
  \end{align*}
  which is a contradiction.
\end{proof}

The reason we study $(R,c)$-expanders is the fact that they entail some
pseudo-random properties, from which we will infer the properties that are
considered in this paper, namely, being Hamiltonian, admitting a perfect
matching, and being $k$-vertex-connected.

\subsection{Properties of random graphs}
In the following technical lemma we establish properties of random
graphs to be used later to prove \cref{thm:gnp}.

\begin{theorem}\label{p16}
  Let $C=150$ and let $C\le\alpha=\alpha(n)\le\frac{n}{\ln{n}}$.  Let
  $p=p(n) = \alpha\cdot\frac{\ln{n}}{n}$, and let $G\sim\gnp$. Then, \whp{},
  \begin{description}
    \item[\namedlabel{P:connected}{(P1)}] $G$ is connected,
    \item[\namedlabel{P:regular}{(P2)}] For every $v\in[n]$,
      $|d(v)-\alpha\ln{n}|\le 2\sqrt{\alpha}\ln{n}$; in particular,
      $\frac{5\alpha}{6}\ln{n}\le d(v) \le \frac{4\alpha}{3}\ln{n}$,
    \item[\namedlabel{P:boundary}{(P3)}] For every non-empty set $S\subseteq
      [n]$ with at most $0.8n$ vertices, $|\partial S| >
      \frac{|S||S^c|\alpha\ln{n}}{2n}$,
    \item[\namedlabel{P:expansion}{(P4)}] For every large enough constant
    $K>0$ (which does not depend on $\alpha$) and for every
    non-empty set $A\subseteq[n]$ with $|A|=a$, the following holds:
      \begin{itemize}
        \item If $\frac{n}{\alpha\ln{n}}\le a \le \frac{n}{\ln{n}}$ then
          \begin{equation*}
            \left|E\left(A,\left\{u\in N(A)\mid |E(u,A)|\ge
              Ka\alpha\ln{n}/n\right\}\right)\right| \le
              a\alpha\ln{n}/K;
          \end{equation*}
        \item If
            $a<\frac{n}{\alpha\ln{n}}$ then
          \begin{equation*}
            \left|E\left((A,\left\{u\in N(A)\mid |E(u,A)|\ge
            K\right\}\right)\right|
            \le a\alpha\ln{n}/\ln{K}.
          \end{equation*}
      \end{itemize}
    \item[\namedlabel{P:bad}{(P5)}] For every set $A$ with
      $|A|=n\left(\ln{\ln{n}}\right)^{1.5}/\ln{n}$ there exist at
      most $|A|/2$ vertices $v$ not in $A$ for which $\left|E(v,A)\right|\le
      \alpha\left(\ln{\ln{n}}\right)^{1.5}/2$,
    \item[\namedlabel{P:spread}{(P6)}] If $\alpha<\ln^2{n}$ then for every
      $v\in[n]$, $0\le r\le\frac{\ln{n}}{15\ln{\ln{n}}}$, and $w\in N(v,r)$
      we have that $|E(w,B(v,r))|\le 5$.
  \end{description}
\end{theorem}

Property \ref{P:connected} is well-known (see, e.g., \cite{FK}), so we omit the
proof here.

\begin{proof}[Proof of \ref{P:regular}]
  We note that $d(v)\sim\bin(n-1,p)$. Denote $\mu=\E{d(v)}=(n-1)p$.  Fix
  a vertex $v\in[n]$. Using Chernoff bounds (\cref{chernoff}) we have
  that
  \begin{align*}
    \pr{d(v) \le \alpha\ln{n} - 2\sqrt{\alpha}\ln{n}}
    &\le \pr{d(v) \le \mu - \sqrt{3\alpha}\ln{n}}\\
    &\le \exp\left(-\frac{3\alpha\ln^2{n}}{2\mu}\right)
    = o\left(n^{-1}\right),
  \end{align*}
  and that
  \begin{align*}
    \pr{d(v) \ge \alpha\ln{n} + 2\sqrt{\alpha}\ln{n}}
    &\le \pr{d(v) \ge \mu + 2\sqrt{\alpha}\ln{n}}\\
    &\le \exp\left(-\frac{4\alpha\ln^2{n}}
      {2\left(\mu+\frac{2}{3}\sqrt{\alpha}\ln{n}\right)}\right)
    \le \exp\left(-\frac{6}{5}\ln{n}\right)
    = o\left(n^{-1}\right).
  \end{align*}
  The union bound over all vertices $v\in [n]$ yields the desired result.
  Since $\alpha\ge 150$ we also ensure that for every $v\in[n]$,
  \begin{equation*}
    \frac{5\alpha}{6}\ln{n} \le d(v) \le \frac{4\alpha}{3}\ln{n}.\qedhere
  \end{equation*}
\end{proof}

\begin{proof}[Proof of \ref{P:boundary}]
  Fix a set $S\subseteq [n]$ with $1 \le |S|=s \le 0.8n$. We note that
  $|\partial S|\sim\bin(s(n-s),p)$, thus by \cref{chernoff} we have that
  \begin{equation*}
    \pr{|\partial S| \le \frac{1}{2}s(n-s)p}
    \le \exp\left(-\frac{1}{8}s(n-s)p\right).
  \end{equation*}
  Let $F$ be the event ``$\exists S$ such that
  $|\partial S|\le\frac{1}{2}s(n-s)p$''.  The union bound gives
  \begin{align*}
    \pr{F} &\le \sum_{s=1}^{0.8n}\binom{n}{s}
    \exp\left(-\frac{1}{8}s(n-s)p\right)\\
    &\le \sum_{s=1}^{0.8n}
    \exp\left(s\left( 1+\ln{n}-\ln{s}-\frac{1}{8}(n-s)p \right)\right)\\
    &\le \sum_{s=1}^{0.8n}
    \exp\left(s\left( 1+\ln{n}-\ln{s}-\frac{\alpha}{40}\ln{n} \right)\right)
    = o(1),
  \end{align*}
  since $\alpha>80$.
\end{proof}

\begin{proof}[Proof of \ref{P:expansion}]
  Fix $A$ with $|A|=a$, and suppose first that $\frac{n}{\alpha\ln{n}}\le
  a\le\frac{n}{\ln{n}}$. Let
  \begin{equation*}
    B_0 = \left\{u\in N(A)\mid |E(b,A)|\ge Kap\right\},
  \end{equation*}
  for large $K$ to be determined later.  For a vertex $u\notin A$, the random
  variable $|E(u,A)|$ is binomially distributed with $a$ trials and success
  probability $p$, and these random variables are independent for different
  vertices $u$.  Thus, using \cref{chernoff:trivial} we have that for large
  enough $K$,
  \begin{equation*}
    \pr{|E(b,A)|\ge Kap}
    \le \left(\frac{e}{K}\right)^{Kap}\le e^{-K}.
  \end{equation*}
  Thus $|B_0|$ is stochastically dominated by a
  binomial random variable with $n$ trials and success probability $e^{-K}$.
  It follows again by \cref{chernoff:trivial} that $\pr{|B_0|>3e^{-K}n} \le
  c^n$ for some $0<c<1$.  Since $a\le n/\ln{n}$,
  $n\binom{n}{a}=o\left(c^{-n}\right)$.  Thus by the union bound,
  \begin{align*}
    \pr{\exists A,\ |A|=a:\ |E(A,B_0)| > \frac{anp}{K}}
    &\le \binom{n}{a}
    \left(c^n + \pr{|E(A,B_0)| > \frac{anp}{K}\mid |B_0| \le 3e^{-K}n}\right)\\
    &\le o\left(n^{-1}\right)
    + \binom{n}{a}\binom{n}{3e^{-K}n}\binom{3ae^{-K}n}{anp/K}p^{anp/K}\\
    &\le o\left(n^{-1}\right)
    + \binom{n}{4e^{-K}n}\binom{3ae^{-K}n}{anp/K}p^{anp/K}\\
    &\le o\left(n^{-1}\right) + e^{4Ke^{-K}n} \cdot \left(9Ke^{-K}\right)^{anp/K}\\
    &= o\left(n^{-1}\right)
    + \left(e^{4Ke^{-K}} \cdot \left(9Ke^{-K}\right)^{ap/K}\right)^n
    = o\left(n^{-1}\right),
  \end{align*}
  for large enough $K$.
  Now suppose $a\le \frac{n}{\alpha\ln{n}}$. Let
  \begin{equation*}
    B_0 = \left\{u\in N(A)\mid |E(b,A)|\ge K\right\}.
  \end{equation*}
  From \ref{P:regular} we know that the number of edges going out from
  $A$ is \whp{} at most $4anp/3$. Given that, $|B_0|\le 2anp/K$. Let $F_a$ be
  the event ``there exists $A$, $|A|=a$, such that
  $|E(A,B_0)|>\frac{anp}{\ln{K}}$''.  Thus,
  \begin{align*}
    \pr{F_a \mid \Delta(G)\le 4np/3}
    &\le\binom{n}{a}
    \binom{n}{2anp/K}\binom{2a^2np/K}{anp/\ln{K}}p^{anp/\ln{K}}\\
    &\le \left[
    n^{1/(np)}
    \left(\frac{eK}{2ap}\right)^{2/K}
    \left(\frac{2eap\ln{K}}{K}\right)^{1/\ln{K}}
    \right]^{anp}
    \\
    &\le \left[
    e^{1/\alpha}
    \left(\frac{eK}{2}\right)^{2/K}
    \left(\frac{2e\ln{K}}{K}\right)^{1/\ln{K}}
    (ap)^{1/\ln{K}-2/K}
    \right]^{anp} = o\left(n^{-1}\right),
  \end{align*}
  for large enough $K$.
  Taking the union bound over all cardinalities $1\le a\le n/\ln{n}$ implies
  that the claim holds \whp{} in both cases.
\end{proof}

\begin{proof}[Proof of \ref{P:bad}]
  Fix a set $A\subseteq[n]$ with
  $|A|=\Lambda=n\left(\ln{\ln{n}}\right)^{1.5}/\ln{n}$. We say that a vertex
  $v\notin A$ is \emph{bad} if $\left|E(v,A)\right|\le\Lambda p/2$.  Since
  $|E(v,A)|\sim\bin(\Lambda,p)$, Chernoff bounds (\cref{chernoff}) give that
  the probability that $v$ is bad with respect to $A$ is at most
  $\exp(-\Lambda p/8)$.

  Let $U_A$ be the set of bad vertices with respect to $A$.  We now show that
  $U_A$ is typically not too large. To this end, note that $|U_A|$ is
  stochastically dominated by a binomial random variable with $n$ trials and
  success probability $\exp(-\Lambda p/8)$. Thus, using \cref{chernoff:trivial},
  we have that
  \begin{equation*}
    \pr{\left|U_A\right|\ge \Lambda/2}
    \le \binom{n}{\Lambda/2}\exp\left(-\Lambda^2 p/16\right).
  \end{equation*}
  The probability that there exists $A$ of cardinality $\Lambda$  whose set
  of bad vertices is of cardinality at least $\Lambda/2$ is thus at most
  \begin{align*}
    \pr{\exists A:\ |A|=\Lambda,\ \left|U_A\right|\ge\Lambda/2}
    &\le \binom{n}{\Lambda}\binom{n}{\Lambda/2}
      \exp\left(-\Lambda^2 p/16\right)\\
    &\le \left(\frac{en}{\Lambda}\right)^{2\Lambda}
      \exp\left(-\Lambda^2 p/16\right)\\
    &\le \exp\left(3\Lambda\ln{(n/\Lambda)}-\Lambda^2p/16\right)\\
    &\le \exp\left(3\cdot\frac{n}{\ln{n}}\left(\ln{\ln{n}}\right)^{2.5}
    -\frac{n}{\ln{n}}\left(\ln{\ln{n}}\right)^3\right)\\
    &\le \exp\left(-\frac{n}{\ln{n}}\left(\ln{\ln{n}}\right)^{2.9}\right)
      = o(1),
  \end{align*}
  here we used $\alpha>16$.
  Noting that $\Lambda p = \alpha\left(\ln{\ln{n}}\right)^{1.5}$, the claim
  follows.
\end{proof}

\begin{proof}[Proof of \ref{P:spread}]
  Assume $\alpha<\ln^2{n}$ and let $\lambda=\frac{\ln{n}}{15\ln{\ln{n}}}$. Fix
  $v\in[n]$, $0\le r\le\lambda$, expose a BFS tree $T$, rooted at $v$, of
  depth $r$, and fix a vertex $w\in N(v,r)$.  Note that in $T$, the vertex $w$
  is a leaf, and thus has a single neighbour in $B(v,r)$.  We now expose the
  rest of the edges between $w$ and $B(v,r)$.  Note that by definition the
  neighbours of $w$ must be in $N(v,r-1)\cup N(v,r)$, and by \ref{P:regular}
  the cardinality of that set is at most $(c\ln^3{n})^{r}$ for some $c>1$, hence
  by \cref{chernoff:trivial},
  \begin{equation*}
    \pr{|E(w,B(v,r))|\ge 6}
    \le \left(\frac{e(c\ln^3{n})^rp}{6}\right)^6
    \le \left(\frac{(c\ln^3{n})^{\lambda+1}}{2n}\right)^6
    = o\left(n^{-3}\right).
  \end{equation*}
  It follows by the union bound that
  \begin{equation*}
    \pr{\exists v\in[n], 0\le r\le\lambda, w\in N(v,r):\
    |E(w,B(v,r))|\ge 6}\\
    \le n\cdot \lambda \cdot n \cdot o\left(n^{-3}\right) = o(1).
    \qedhere
  \end{equation*}
\end{proof}

For $\alpha=\alpha(n)>0$, a graph for which
\ref{P:connected},\ldots,\ref{P:spread} hold will be called
\emph{$\alpha$-pseudo-random}.

\subsection{Properties of random walks}
Throughout this section, $G$ is a graph with vertex set $[n]$, having
properties \ref{P:connected}, \ref{P:regular} and \ref{P:boundary}
for some $\alpha>0$, and $X$ is a $\frac{1}{2}$-\emph{lazy} simple random walk
on $G$, starting at some arbitrary vertex $v_0$.  By $\frac{1}{2}$-lazy we mean
that it stays put with probability $\frac{1}{2}$ at each time step, and moves
to a uniformly chosen random neighbour otherwise.
Our purpose in this section is to show that
$X$ mixes well, in a sense that will be further clarified below. To this
end, we shall need some preliminary definitions and notations.

The \emph{transition probability} of $X$ from $u$ to $v$ is the probability
\begin{equation*}
  p_{uv} = \pr{X_{t+1}=v\mid X_t=u} =
  \pr{X_1=v\mid X_0=u}.
\end{equation*}
For $k\in\mathbb{N}$ we similarly denote
\begin{equation*}
  p_{uv}^{(k)} = \pr{X_{t+k}=v\mid X_t=u} =
  \pr{X_k = v\mid X_0 = u}.
\end{equation*}
We note that the stationary distribution of $X$ is given by
\begin{equation*}
  \pi_v = \frac{d(v)}{\sum_{u\in [n]}d(u)} = \frac{d(v)}{2\left|E\right|},
\end{equation*}
and for every subset $S\subseteq[n]$,
\begin{equation*}
  \pi_S = \sum_{v\in S}\pi_v.
\end{equation*}
The \emph{total variation distance} between $X_t$ and the stationary
distribution is
\begin{equation*}
  \dtv{X_t}{\pi} = \sup_{S\subseteq [n]}\left|\pr{X_t\in S} - \pi_S\right|,
\end{equation*}
and as is well-known, we have that
\begin{equation*}
  \dtv{X_t}{\pi} = \frac{1}{2}\sum_{v\in[n]}\left|\pr{X_t=v}-\pi_v\right|.
\end{equation*}

Now, let $\left(Y_t\right)_{t\ge 0}$ be the \emph{stationary walk} on $G$; that
is, the $\frac{1}{2}$-lazy simple random walk for which for every $v\in[n]$,
$\pr{Y_0=v}=\pi_v$. We note for later use that there exists a standard coupling
of $X,Y$ under which for every $t$,
\begin{equation*}
  \pr{\exists s\ge t\mid X_s\ne Y_s} \le \dtv{X_t}{\pi}.
\end{equation*}
Our goal is therefore to find not too large $t$'s for which the total
variation distance is very small. That is, we wish to bound the
\emph{$\xi$-mixing time} of $X$, which is given by
\begin{equation*}
  \tau(\xi)
  = \min\left\{t\ge 0\mid \forall s\ge t,\
    \dtv{X_s}{\pi}<\xi\right\}.
\end{equation*}
A theorem of Jerrum and Sinclair~\cite{JS89} will imply that the
$\xi$-mixing time of $X$ is indeed small. Their bound uses the notion
of \emph{conductance}: the conductance of a cut $(S,S^c)$ with respect to
$X$ is defined as
\begin{equation*}
  \phi_X(S) = \frac{\sum_{v\in S,\ w\in S^c}\pi_v p_{vw}}
  {\min\left(\pi_S,\pi_{S^c}\right)},
\end{equation*}
which can be equivalently written in our case as
\begin{equation*}
  \phi_X(S) = \frac{\left|\partial{S}\right|}
  {2\min\left(\sum_{v\in S}d(v),\sum_{w\in S^c}d(w)\right)}.
\end{equation*}
The conductance of $G$ with respect to $X$ is defined as
\begin{equation*}
  \Phi_X(G) = \min_{\substack{S\subseteq [n]\\0<\pi_S\le1/2}} \phi_X(S).
\end{equation*}

\begin{claim}\label{jerrum}
  Let $\pi_{\min} = \min_{v}\pi_v$. For every $\xi>0$,
  \begin{equation*}
    \tau(\xi) \le\frac{2}{\Phi_X(G)^2}
    \left(\ln\left(\frac{1}{\pi_{\min}}\right)
    +\ln\left(\frac{1}{\xi}\right)\right).
  \end{equation*}
\end{claim}

\begin{proof}
  Let
  \begin{equation*}
    \tau'(\xi) = \min\left\{t\ge 0\mid \forall s\ge t,\
      u,v\in[n],\ \frac{\left|p_{uv}^{(s)}
      -\pi_v\right|}{\pi_v}<\xi\right\}
  \end{equation*}
  be the $\xi$-uniform mixing time of $X$.
  Noting that the laziness of the walk implies $p_{uu}\ge\frac{1}{2}$ for every
  vertex $u$, Corollary 2.3 in~\cite{JS89} states that
  \begin{equation*}
    \tau'(\xi) \le\frac{2}{\Phi_X(G)^2}
    \left(\ln\left(\frac{1}{\pi_{\min}}\right)
    +\ln\left(\frac{1}{\xi}\right)\right).
  \end{equation*}
  Let $t=\tau'(\xi)$; thus, for all $s\ge t$, $u,v\in[n]$,
  $\left|p_{uv}^{(s)}-\pi_v\right|<\xi\pi_v$. Fix $s\ge t$. We have
  that
  \begin{equation*}
    \dtv{X_s}{\pi}
    = \frac{1}{2}\sum_{v\in[n]}\left|\pr{X_s=v}-\pi_v\right|\\
    \le \frac{\xi}{2}\sum_{v\in[n]}\pi_v = \frac{\xi}{2},
  \end{equation*}
  thus $\tau(\xi)\le \tau(\xi/2)\le t=\tau'(\xi)$ and the claim follows.
\end{proof}

\begin{corollary}\label{jerrum_cor}
  For $\xi>0$, $\tau(\xi)\le 1800\ln(2n/\xi)$.
\end{corollary}

\begin{proof}
  We note that due to \ref{P:regular}, for every $v\in[n]$,
  \begin{equation*}
    \pi_v\ge \frac{5}{8n},
  \end{equation*}
  and thus for every $S\subseteq[n]$ with $0<\pi_S\le 1/2$ we have that
  \begin{equation*}
    \frac{1}{2} \ge \pi_S \ge |S|\cdot\frac{5}{8n},
  \end{equation*}
  hence $0<|S|\le\frac{4}{5}n$. Thus, according to
  \ref{P:regular},\ref{P:boundary},
  \begin{align*}
    \Phi_X(G) &= \min_{\substack{S\subseteq [n]\\0<\pi_S\le1/2}} \phi_X(S)\\
    &\ge \min_{\substack{S\subseteq [n]\\0<|S|\le4n/5}}
    \frac{|\partial S|}{2\sum_{v\in S}d(v)}\\
    &\ge \min_{\substack{S\subseteq [n]\\0<|S|\le4n/5}}
    \frac{\frac{|S||S^c|\alpha\ln{n}}{2n}}
    {2|S|\cdot \frac{4}{3}\alpha\ln{n}}\ge\frac{1}{30}.
  \end{align*}
  Plugging this into \cref{jerrum} we have
  \begin{equation*}
    \tau(\xi) \le 1800\left(\ln\left(\frac{8n}{5}\right)
    +\ln\left(\frac{1}{\xi}\right)\right)
    \le 1800\ln\left(\frac{2n}{\xi}\right).\qedhere
  \end{equation*}
\end{proof}

The following is an immediate corollary of the above discussion:
\begin{corollary}\label{buffer}
  Let $b = \tau(1/n) \le 3601\ln{n}$.  Conditioned on $X_0,\ldots,X_t$,
  there exists a coupling of
  $\left(X_{t+b+s}\right)_{s\ge 0}$ and $\left(Y_s\right)_{s\ge 0}$ under
  which
  \begin{equation*}
    \pr{\exists s\ge 0\mid X_{t+b+s}\ne Y_{s}} \le \frac{1}{n}.
  \end{equation*}
\end{corollary}

\section{Walking on a pseudo-random graph}\label{sec:pseudo}
In order to prove \cref{thm:gnp}, we will prove that the trace
$\Gamma=\Gamma^v_L(G)$ is \whp{} a good expander, in the sense that it
satisfies the following two properties:
\begin{description}
  \item[\namedlabel{E:small}{(E1)}]
    There exists $\beta>0$ such that every set $A\subseteq[n]$ of
    cardinality $|A|\le\frac{n}{\ln{n}}$ satisfies $|N_\Gamma(A)|\ge
    |A|\cdot \beta\ln{n}$;
  \item[\namedlabel{E:large}{(E2)}]
    There is an edge of $\Gamma$ between every pair of disjoint subsets
    $A,B\subseteq[n]$ satisfying $|A|,|B|\ge
    \frac{n\left(\ln{\ln{n}}\right)^{1.5}}{\ln{n}}$.
\end{description}

\begin{theorem}\label{blackbox}
  For every $\varepsilon>0$ there exist $C=C(\varepsilon)>0$
  and $\beta=\beta(\varepsilon)>0$ such that for
  every edge probability $p=p(n)\ge C\cdot \frac{\ln{n}}{n}$ and for every
  $v_0\in[n]$, a random graph $G\sim\gnp$ is \whp{} such that for
  $L=(1+\varepsilon)n\ln{n}$, the trace $\Gamma^{v_0}_L(G)$ of a simple random
  walk of length $L$ on $G$, starting at $v_0$, has the properties
  \ref{E:small} and \ref{E:large} \whp{}.
\end{theorem}

It will be convenient the consider a slight variation of this theorem, in
which the random walk is \emph{lazy} and the base graph is pseudo-random:
\begin{theorem}\label{lazy_version}
  For every sufficiently small $\varepsilon>0$, if
  $\alpha\ge 1500\varepsilon^{-2}$
  and $G$ is a $\alpha$-pseudo-random graph on the
  vertex set $[n]$, $v_0\in[n]$ and $L_2=(2+\varepsilon)n\ln{n}$,
  then the trace $\Gamma_{L_2}^{v_0}(G)$ of a $\frac{1}{2}$-lazy
  random walk of length $L_2$ on $G$, starting at $v_0$, has the
  properties \ref{E:small} and \ref{E:large} \whp{}.
\end{theorem}

Before proving this theorem, we show that \cref{blackbox} is a simple
consequence of it.

\begin{proof}[Proof of \cref{blackbox}]
  Since \ref{E:small} and \ref{E:large} are monotone, we may assume
  $\varepsilon$ is sufficiently small.  Let $L=(1+\varepsilon)n\ln{n}$,
  $L_2=(2+\varepsilon)n\ln{n}$, $C=1500\varepsilon^{-2}$ and
  $p=\alpha\ln{n}/n$ for $\alpha\ge C$.
  Let $(X_t^{v_0})_{t=0}^{L_2}$ be the $\frac{1}{2}$-lazy random
  walk of length $L_2$ on $G$, starting at $v_0$, and define
  \begin{equation*}
    R = \left|\left\{0 < t \le L_2 \mid X_{t}\ne X_{t-1}\right\}\right|.
  \end{equation*}
  Since $\pr{X_t=X_{t-1}}=1/2$ for every $0<t\le L_2$, by Chernoff bounds
  (\cref{chernoff}) we have that $\pr{R>L}=o(1)$.  Denote by
  $\Gamma^\ell_{L_2}$ the trace of that walk, ignoring any loops, and let $P$
  be a monotone graph property which $\Gamma^\ell_{L_2}$ satisfies \whp{}.
  Given $R$, the trace $\Gamma^\ell_{L_2}$ has the same distribution as the
  trace of the non-lazy walk $\Gamma_R$.  Thus:
  \begin{equation*}
    \pr{\Gamma_L\in P}
    \ge \pr{\Gamma^\ell_{L_2}\in P, R\le L} = 1-o(1).
  \end{equation*}
  As \ref{E:small} and \ref{E:large} are both monotone, and since $G$
  is $\alpha$-pseudo-random \whp{} by \cref{p16}, the claim holds using
  \cref{lazy_version}.
\end{proof}

Thus, in what follows, $\varepsilon>0$ is sufficiently small, $G$ is a
$\alpha$-pseudo-random graph on the vertex set $[n]$ for $\alpha\ge
C=1500\varepsilon^{-2}$, $v_0\in[n]$ is some fixed vertex,
$L_2=(2+\varepsilon)n\ln{n}$, $(X_t^{v_0})_{t=0}^{L_2}$ is a $\frac{1}{2}$-lazy
simple random walk of length $L_2$ on $G$, starting at $v_0$ (which we may
simply refer to as $X$), and $(Y_t)_{t=0}^\infty$ (or simply $Y$) is the
$\frac{1}{2}$-lazy simple random walk on $G$, starting at a random vertex sampled
from the stationary distribution of $X$ (the \emph{stationary walk}).

\bigskip

The rest of this section is organised as follows. In the first subsection we
show that \whp{}, every vertex is visited at least a logarithmic number of
times. In the second and third subsections we use this fact to conclude that
small vertex sets typically expand well, and that large vertex sets
are typically connected, by that proving that the trace satisfies
\ref{E:small} and \ref{E:large}  \whp{}.

\subsection{Number of visits}
Define
\begin{equation*}
  \nu(v) = \left|\left\{0<t\le L_2\mid X_t=v,\ X_{t+1}\ne v\right\}\right|,
  \quad v\in[n].
\end{equation*}

\begin{theorem}\label{visits}
  There exists $\rho>0$ such that \whp{}, for every $v\in [n]$, $\nu(v)\ge
  \rho\ln{n}$.
\end{theorem}

In order to prove this theorem, we first introduce a number of definitions and
lemmas.  Recall that a \emph{supermartingale} is a sequence $M(0),M(1),\ldots$
of random variables such that each conditional expectation $\E{M(t+1)\mid
M(0),\ldots,M(t)}$ is at most $M(t)$.  Given such a sequence, a \emph{stopping
rule} is a function from finite histories of the sequence into
$\left\{0,1\right\}$, and a \emph{stopping time} is the minimum time at which
the stopping rule is satisfied (that is, equals $1$). For two integers $s,t$,
let $s\wedge t=\min\left\{s,t\right\}$.  Let
$\lambda=\frac{\ln{n}}{15\ln{\ln{n}}}$, and for every $v\in[n]$ let $F_t^v$ be
the event ``$Y_t=v$ or $d_G(Y_t,v)>\lambda$'' (recall that for two vertices
$u,v$, $d_G(u,v)$ denotes the distance from $u$ to $v$ in $G$).
Note that if $\alpha<\ln^2{n}$ it follows from \ref{P:regular} and
\ref{P:spread} that the diameter of $G$ is larger than $\lambda$.

\begin{lemma}\label{superm}
  Suppose $\alpha<\ln^2{n}$. For $v\in[n]$, define the process
  \begin{equation*}
    \mathcal{M}^v(t)=\left(\frac{10}{\alpha\ln{n}}\right)^{d_G(Y_t,v)}.
  \end{equation*}
  Let $S=\min\left\{t\ge 0: F_t^v\right\}$ be a stopping time; then
  $\mathcal{M}^v(t\wedge S)$ is a supermartingale. In particular, for every
  $u\in[n]$ the stationary walk $Y_t$ satisfies
  \begin{equation*}
    \pr{Y_S=v\mid Y_0=u}
    \le \left(\frac{10}{\alpha\ln{n}}\right)^{d_G(u,v)}.
  \end{equation*}
\end{lemma}

\begin{proof}
  For a vertex $w\in[n]$, denote
  \begin{align*}
    p_\leftarrow(w) &= \pr{d_G(Y_1,v) < d_G(Y_0,v) \mid Y_0=w},\\
    p_\rightarrow(w) &= \pr{d_G(Y_1,v) > d_G(Y_0,v) \mid Y_0=w}.
  \end{align*}
  We note that for $0<y\le x\le 1$, $\frac{y}{x}+x-y\le 1$. Thus, for
  $q_1,q_2>0$ for which
  $\frac{p_\leftarrow(w)}{p_\rightarrow(w)}\le\frac{q_1}{q_2}\le 1$,
  \begin{align*}
    &\E{\left(\frac{q_1}{q_2}\right)^{d_G(Y_1,v)}\mid Y_0=w}
    -\left(\frac{q_1}{q_2}\right)^{d_G(w,v)}\\
    &=
    \left(\frac{q_1}{q_2}\right)^{d_G(w,v)}
    \left(
    \left(p_\leftarrow(w)\frac{q_2}{q_1}+p_\rightarrow(w)\frac{q_1}{q_2}
    +\left(1-p_\leftarrow(w)-p_\rightarrow(w)\right)\right)
    -1\right)\\
    &=
    \left(\frac{q_1}{q_2}\right)^{d_G(w,v)}p_\rightarrow(w)
    \left(\frac{bp_\leftarrow(w)}{ap_\rightarrow(w)}
    + \frac{q_1}{q_2} - \frac{p_\leftarrow(w)}{p_\rightarrow(w)} - 1\right)
    \le 0.
  \end{align*}
  Let $w$ be such that $0<d_G(v,w)\le\lambda$. Since $\alpha<\ln^2{n}$,
  $G$ satisfies \ref{P:spread}. Considering that and \ref{P:regular}, and since
  $\alpha>25$, we have that
  \begin{align*}
    q_\leftarrow &:= \frac{5}{2(\alpha-2\sqrt{\alpha})\ln{n}}
    \ge \frac{5}{2d_G(w)} \ge p_\leftarrow(w),\\
    q_\rightarrow &:= \frac{\alpha\ln{n}}{4(\alpha-2\sqrt{\alpha})\ln{n}}
    \le \frac{1}{2} - \frac{5}{2(\alpha-2\sqrt{\alpha})\ln{n}}
    \le \frac{d_G(w)-5}{2d_G(w)} \le p_\rightarrow(w),
  \end{align*}
  and as
  $\frac{p_\leftarrow(w)}{p_\rightarrow(w)}
  \le \frac{q_\leftarrow}{q_\rightarrow}=\frac{10}{\alpha\ln{n}} \le 1$ we have
  that $\mathcal{M}^v(t\wedge S)$ is a supermartingale.  In addition, for every
  $t\ge 0$, almost surely,
  \begin{align*}
    \left(\frac{10}{\alpha\ln{n}}\right)^{d_G(Y_0,v)}
    &= \mathcal{M}^v(0)\\
    &\ge \E{\mathcal{M}^v(t\wedge S)\mid Y_0}\\
    &= \sum_{w\in[n]}\left(\frac{10}{\alpha\ln{n}}\right)^{d_G(w,v)}
    \cdot\pr{Y_{t\wedge S}=w\mid Y_0}\\
    &\ge \pr{Y_{t\wedge S}=v\mid Y_0}.
  \end{align*}
  As this is true for every $t\ge 0$, and since $S$ is almost surely finite, it
  follows that for every $u\in[n]$,
  \begin{equation*}
    \pr{Y_S=v\mid Y_0=u} \le
    \left(\frac{10}{\alpha\ln{n}}\right)^{d_G(u,v)}
  \end{equation*}
  \whp{}.
\end{proof}

Let
\begin{equation*}
  T=\ln^2{n}.
\end{equation*}
For a walk $W$ on $G$, let $I_W(v,t)$ be the event ``$W_t=v$ and $W_{t+1}\ne
v$''. Our next goal is to estimate the probability that $I_Y(v,t)$ occurs for
some $1\le t< T$, given that $I_Y(v,0)$ has occurred.

\begin{lemma}\label{comeback}
  For every vertex $v\in[n]$ we have that
  \begin{equation*}
    \pr{\bigcup_{1\le t< T} I_Y(v,t)\mid I_Y(v,0)} \le
    \ln^{-1/2}{n}.
  \end{equation*}
\end{lemma}

\begin{proof}
  Fix $v\in[n]$.  Define the following sequence of stopping times: $U_0=0$,
  and for $i\ge 1$,
  \begin{equation*}
    U_i = \min\left\{t>U_{i-1}\mid F_t^v\right\}.
  \end{equation*}
  Then,
  \begin{equation*}
    \pr{\bigcup_{1\le t< T}I_Y(v,t)\mid I_Y(v,0)}
    \le \pr{Y_{U_1}=v\mid I_Y(v,0)}
    + \sum_{i=2}^{T-1} \pr{Y_{U_i}=v\mid Y_{U_{i-1}}\ne v,\ I_Y(v,0)}.
  \end{equation*}
  Now, if $\alpha<\ln^2{n}$, \cref{superm} and the Markov property imply
  that
  \begin{align*}
    \pr{Y_{U_1}=v\mid I_Y(v, 0)} &\le \frac{10}{\alpha \ln{n}},\\
    \pr{Y_{U_i}=v\mid Y_{U_{i-1}}\ne v,\ I_Y(v, 0)}
    &\le \left(\frac{10}{\alpha\ln{n}}\right)^\lambda
    \qquad (i\ge 2),
  \end{align*}
  so
  \begin{align*}
    \pr{\bigcup_{1\le i< T}I_Y(v,t)\mid I_Y(v,0)}
    &\le \frac{10}{\alpha\ln{n}} +
    T\cdot\left(\frac{10}{\alpha\ln{n}}\right)^\lambda\\
    &= \frac{10}{\alpha\ln{n}} + \ln^2{n}\cdot o\left(n^{-1/20}\right)
    \le \frac{20}{\alpha\ln{n}}
    \le \ln^{-1/2}{n}.
  \end{align*}

  Now consider the case $\alpha\ge\ln^2{n}$. The number of exits from $v$ at
  times $1,\ldots,T-1$ is at most the number of entries to $v$ at times
  $1,\ldots,T-2$ plus $1$. Recalling \ref{P:regular}, at any time
  $i\in[T-3]$, the probability to enter $v$ at time $i+1$ is at most
  $1/d_G(X_i) \le \ln^{-2.5}{n}$. Thus, the number of exits from $v$ is
  stochastically dominated by a binomial random variable with $T-1$ trials
  and success probability $\ln^{-2.5}{n}$. Thus,
  \begin{align*}
    \pr{\bigcup_{1\le i< T}I_Y(v,t)\mid I_Y(v,0)}
    &\le \pr{\sum_{t=1}^{T-1}\mathbf{1}_{I_Y(v,t)}\ge 1}\\
    &\le (T-1)\ln^{-2.5}{n} \le \ln^{-1/2}{n},
  \end{align*}
  and the claim follows.
\end{proof}

For a walk $W$ on $G$ and a vertex $v\in[n]$, let
$M_W(v)=\sum_{t=0}^{T-1}\mathbf{1}_{I_W(v,t)}$.

\begin{lemma}\label{visits_stat}
  For every vertex $v\in[n]$, $\pr{M_Y(v) \ge 1}\ge
  \frac{T}{2n}(1-6\alpha^{-1/2})$.
\end{lemma}

\begin{proof}
  Fix $v\in[n]$.  It follows from \ref{P:regular} that
  \begin{align*}
    \E{M_Y(v)}
    = \sum_{t=0}^{T-1}\pr{I_Y(v,t)}
    &= T\cdot\pr{Y_0=v\land Y_1\ne v}
    = T\cdot\pr{Y_0=v}\pr{Y_1\ne Y_0}\\
    &= \frac{T\cdot \pi_v}{2}
    = \frac{T\cdot d(v)}{4|E(G)|}
    \ge \frac{T\cdot \left(1-2\alpha^{-1/2}\right)}
    {2n\cdot \left(1+2\alpha^{-1/2}\right)}
    \ge \frac{T}{2n}\cdot \left(1-5\alpha^{-1/2}\right).
  \end{align*}
  Thus by \cref{comeback},
  \begin{align*}
    \frac{T}{2n}\left(1-5\alpha^{-1/2}\right)
    &\le \E{M_Y(v)}\\
    &= \sum_{i=1}^\infty i\pr{M_Y(v)=i}\\
    &\le \pr{M_Y(v)=1}\sum_{i=1}^\infty
    i\left(\ln^{-1/2}{n}\right)^{i-1}\\
    &= \pr{M_Y(v)=1}\left(1-\left(\ln^{-1/2}{n}\right)\right)^{-2}.
  \end{align*}
  It follows that
  \begin{equation*}
    \pr{M_Y(v)\ge 1}
    \ge \pr{M_Y(v) = 1}
    \ge \frac{T}{2n}\left(1-6\alpha^{-1/2}\right).\qedhere
  \end{equation*}
\end{proof}

Let
\begin{equation*}
  b=3601\ln{n}.
\end{equation*}

\begin{corollary}\label{visit_segment}
  Let $t\ge 0$. Conditioned on $X_0,\ldots,X_t$, for every vertex $v\in [n]$,
  \begin{equation*}
    \pr{M_{\left(X_{t+b+s}\right)_{s\ge 0}}(v) \ge 1} \ge
      \frac{T}{2n}(1-6\alpha^{-1/2}) - \frac{1}{n}.
  \end{equation*}
\end{corollary}

\begin{proof}
  According to \cref{visits_stat,buffer},
  \begin{equation*}
    \pr{M_{\left(X_{t+b+s}\right)_{s\ge 0}}(v) \ge 1}
    \ge \pr{M_{\left(Y_s\right)_{s\ge 0}}(v)\ge 1} - \frac{1}{n}
    \ge \frac{T}{2n}\left(1-6\alpha^{-1/2}\right)-\frac{1}{n}.\qedhere
  \end{equation*}
\end{proof}

\begin{proof}[Proof of \cref{visits}]
  Consider dividing the $L_2$ steps of the walk $X$ into segments of length
  $T+1$ with buffers of length $b$ between them (and before the first).
  Formally, the $k$'th segment is the walk
  \begin{equation*}
    \left(X^{(k)}_s\right)_{s=0}^T = \left(X_{(k-1)(T+1)+kb+s}\right)_{s=0}^T.
  \end{equation*}
  It follows from \cref{visit_segment} that independently between
  the segments, for a given $v$,
  \begin{equation*}
    \pr{M_{X^{(k)}}(v)\ge 1} \ge \frac{T}{2n}\left(1-6\alpha^{-1/2}\right)
    -\frac{1}{n}.
  \end{equation*}
  Thus, $\nu(v)$ stochastically dominates a binomial random variable with
  $\floor{L_2/(T+1+b)}$ trials and success probability
  $\frac{T}{2n}\left(1-6\alpha^{-1/2}\right)-\frac{1}{n}$.  Let
  \begin{equation*}
  \mu=\floor{L_2/(T+1+b)}\cdot
  \left(\frac{T}{2n}\left(1-6\alpha^{-1/2}\right)-\frac{1}{n}\right),
  \end{equation*}
  and note that
  \begin{equation*}
    \mu \sim \frac{2+\varepsilon}{2}\ln{n}\cdot \left(1-6\alpha^{-1/2}\right).
  \end{equation*}
  Thus, letting $\eta=(1-6\alpha^{-1/2})(1+\varepsilon/2)$, by
  \cref{chernoff:cor},
  \begin{align*}
    \pr{\nu(v)<\rho\ln{n}}
    = \pr{\nu(v)<\frac{\rho}{\eta}\cdot \mu(1+o(1))}
    &\le \exp\left(-\mu(1+o(1))(1-e^{-c}-\rho c/\eta)\right)\\
    &\le \exp\left(-\ln{n}(1+o(1)) (\eta(1-e^{-c}) - \rho c) \right),
  \end{align*}
  for every $c>0$.  Since $\alpha\ge 1500\varepsilon^{-2}$ we have that
  $\eta\ge 1+\varepsilon/3$, and taking $c>\ln(13/\varepsilon)$
  gives $\eta(1-e^{-c})>1+\varepsilon/4$, and finally taking
  $\rho<\varepsilon/20c$ gives
  \begin{equation*}
    \pr{\nu(v)<\rho\ln{n}} = o\left(n^{-1+\varepsilon/5}\right).
  \end{equation*}
  The union bound yields
  \begin{equation*}
    \pr{\exists v\in[n]:\ \nu(v)<\rho\ln{n}}
    \le n\cdot\pr{\nu(v)<\rho\ln{n}} = o(1),
  \end{equation*}
  and that concludes the proof.
\end{proof}

\subsection{Expansion of small sets in the trace graph}
Let $\Gamma=\Gamma_{L_2}^{v_0}(G)$ be the trace of $X$.

\begin{theorem}\label{trace:expansion:1}
  There exists $\beta>0$ such that \whp{} for every set $A\subseteq [n]$ with
  $|A|=a\le n/\ln{n}$, it holds that $|N_{\Gamma}(A)|\ge \beta \cdot
  a\ln{n}$.
\end{theorem}

\begin{proof}
  We prove the theorem assuming that for every $v\in[n]$, $\nu(v)\ge\rho\ln{n}$,
  and recall that this event occurs, according to \cref{visits}, \whp{}.
  Let $K>0$ be large enough so that \ref{P:expansion} holds.  Suppose first
  that $A$ is such that $a \ge \frac{n}{\alpha\ln{n}}$. Let
  \begin{equation*}
    B_0 = \left\{u\in N_G(A)\midd|E_G(u,A)|\ge Ka\alpha\ln{n}/n\right\},
  \end{equation*}
  and let
  \begin{equation*}
    A_0 = \left\{v\in A\midd|E_G(v,B_0)| \ge 2\alpha\ln{n}/K\right\}.
  \end{equation*}
  According to \ref{P:expansion}, $|A_0|\le\frac{a}{2}$.  Let
  $A_1=A\smallsetminus A_0$.  For a vertex $v\in A_1$, let $\gamma(v)$ count
  the number of moves (if any) the walk has made from $v$ to $B_0$ in the first
  $\rho\ln{n}$ exits from $v$. By \nameref{P:regular},
  \begin{equation*}
    \pr{\gamma(v)\ge\frac{\rho\ln{n}}{2}}
    \le \binom{\rho\ln{n}}{\rho\ln{n}/2}
    \left(\frac{|E_G(v,B_0)|}{d_G(v)}\right)^{\rho\ln{n}/2}\le
    \left(\frac{15}{K}\right)^{\rho\ln{n}/2}.
  \end{equation*}
  Let $A_2\subseteq A_1$ be the set of vertices $v\in A_1$ with $\gamma(v) \ge
  \rho\ln{n}/2$. We have that
  \begin{equation*}
    \pr{|A_2| \ge \frac{a}{4}}
    \le \binom{a}{a/4}
    \pr{\gamma(v)\ge\frac{\rho\ln{n}}{2}}^{a/4}
    \le \left(\frac{16}{K}\right)^{a\rho\ln{n}/8}.
  \end{equation*}
  Note that for large enough $K$,
  $\binom{n}{a}\left(\frac{16}{K}\right)^{a\rho\ln{n}/8} =
  o\left(n^{-1}\right)$.  Set $B_1=N_G(A)\smallsetminus B_0$,
  $B_2=N_G^+(A)\smallsetminus B_0$ and
  $A_3=A_1\smallsetminus A_2$.  Fix $B_3\subseteq B_1$.
  For $v\in A_3$, let $p_v$ be the probability that a walk which exits $v$
  and lands in $B_2$ will also land in $A\cup B_3$. By \nameref{P:regular} we have
  that for every $v\in A_3$,
  \begin{equation*}
    p_v \le \frac{|E_G(v,A\cup B_3)|}{|E_G(v,B_2)|}
    \le \frac{|E_G(v,B_3)|+a}{\frac{5}{6}\alpha\ln{n} - |E_G(v,B_0)|}
    \le \frac{|E_G(v,B_3)|+a}{\frac{4}{5}\alpha\ln{n}},
  \end{equation*}
  for large enough $K$.
  Assuming that $|A_3|\ge\frac{a}{4}$, we have that
  $|E_G(A_3,B_3)|\le |B_3|\cdot Ka\alpha\ln{n}/n$, hence
  \begin{equation*}
    \frac{1}{|A_3|}\sum_{v\in A_3}p_v
    \le \frac{4}{a}\cdot \frac{|E_G(A_3,B_3)|+a}{\frac{4}{5}\alpha\ln{n}}
    \le \frac{5K|B_3|}{n} + \frac{5}{\alpha\ln{n}},
  \end{equation*}
  hence by the inequality between the arithmetic and geometric means,
  \begin{equation*}
    \prod_{v\in A_3} p_v
    \le \left(\frac{5K|B_3|}{n} + \frac{5}{\alpha\ln{n}} \right)^{|A_3|}.
  \end{equation*}
  In particular, if $N_\Gamma(A)\subseteq B$ for $B$ with $|B|\le a\beta\ln{n}$
  then there exists $B_3\subseteq B_1$ with $|B_3|\le a\beta\ln{n}$ such that
  $N_\Gamma(A_3)\cap B_2 \subseteq A\cup B_3$, and conditioning on
  $|A_3|\ge a/4$ and making sure $\beta=\beta(K)$ is small enough, the
  probability of that event is at most
  \begin{equation*}
    \prod_{v\in A_3}(p_v)^{\rho\ln{n}/2}
    \le \left(\frac{5Ka\beta\ln{n}}{n}
              + \frac{5}{\alpha\ln{n}} \right)^{a\rho\ln{n}/8}
    \le \left(\frac{6Ka\beta\ln{n}}{n}\right)^{a\rho\ln{n}/8}.
  \end{equation*}
  Thus, taking the union bound,
  \begin{align*}
    &\pr{\exists A,\ |A|=a:\ \left|N_\Gamma(A)\right|\le a\beta\ln{n}}\\
    &\le \sum_{\substack{A\\|A|=a}}
      \pr{\left|N_\Gamma(A)\right|\le a\beta\ln{n}}\\
    &\le \sum_{\substack{A\\|A|=a}}\left[
      \pr{|A_2|\ge \frac{a}{4}}
      + \pr{\exists B,|B|=a\beta\ln{n},N_\Gamma(A)\subseteq B \mid
          \left|A_3\right|\ge \frac{a}{4}}
    \right]\\
    &\le \binom{n}{a}\left(\frac{16}{K}\right)^{a\rho\ln{n}/8}
      + \sum_{\substack{A\\|A|=a}} \sum_{\substack{B\\|B|=a\beta\ln{n}}}
      \pr{N_\Gamma(A)\subseteq B\mid |A_3|\ge \frac{a}{4}}\\
    &\le o\left(n^{-1}\right)
      + \binom{n}{a}\binom{n}{a\beta\ln{n}}
        \left(\frac{6Ka\beta\ln{n}}{n}\right)^{a\rho\ln{n}/8}\\
    &\le o\left(n^{-1}\right)
      + \left[\left(\frac{en}{a\beta\ln{n}}\right)^{2\beta}
        \left(\frac{6Ka\beta\ln{n}}{n}\right)^{\rho/8}
      \right]^{a\ln{n}},
  \end{align*}
  and we may take $\beta>0$ to be small enough so that the above expression
  will tend to $0$ faster than $1/n$.

  Now consider the case where $a<\frac{n}{\alpha\ln{n}}$. Let
  \begin{equation*}
    B_0 = \left\{u\in N_G(A)\midd|E_G(u,A)|\ge K\right\},
  \end{equation*}
  and
  \begin{equation*}
    A_0 = \left\{v\in A\midd
      \left|E_G(v,B_0)\right|\ge 2\alpha\ln{n}/\ln{K}\right\}.
  \end{equation*}
  According to \ref{P:expansion}, $|A_0|\le\frac{a}{2}$. Let
  $A_1=A\smallsetminus A_0$.  Define $\gamma(v)$ for vertices from $A_1$ as
  in the first case.  It follows that
  \begin{equation*}
    \pr{\gamma(v) \ge \frac{\rho\ln{n}}{2}}
    \le \binom{\rho\ln{n}}{\rho\ln{n}/2}
    \left(\frac{\left|E_G(v,B_0)\right|}{d_G(v)}\right)^{\rho\ln{n}/2}
    \le \left(\frac{15}{\ln{K}}\right)^{\rho\ln{n}/2}.
  \end{equation*}
  Let $A_2\subseteq A_1$ be the set of vertices $v$ with $\gamma(v) \ge
  \rho\ln{n}/2$. We have that
  \begin{equation*}
    \pr{|A_2| \ge \frac{a}{4}}
    \le \binom{a}{a/4}\pr{\gamma(v)\ge
      \frac{\rho\ln{n}}{2}}^{a/4}
    \le \left(\frac{16}{\ln{K}}\right)^{a\rho\ln{n}/8}.
  \end{equation*}
  Note that for large enough $K$,
  $\binom{n}{a}\left(\frac{16}{\ln{K}}\right)^{a\rho\ln{n}/8} =
  o\left(n^{-1}\right)$.  Set $B_1=N_G(A)\smallsetminus B_0$,
  $B_2=N_G^+(A)\smallsetminus B_0$ and
  $A_3=A_1\smallsetminus A_2$.  Fix $B_3\subseteq B_1$.  For $v\in A_3$, let
  $p_{v}$ be the probability that a walk which exits $v$ and lands in $B_2$
  will also land in $A\cup B_3$. We have that for every $v\in A_3$,
  \begin{equation*}
    p_v \le \frac{|E_G(v,A\cup B_3)|}{|E_G(v,B_2)|}
    \le \frac{|E_G(v, B_3)|+a}{\frac{5}{6}\alpha\ln{n}-|E_G(v,B_0)|}
    \le \frac{|E_G(v, B_2)|+a}{\frac{4}{5}\alpha\ln{n}},
  \end{equation*}
  For large enough $K$.  Assuming that $|A_3|\ge\frac{a}{4}$, we have that
  $|E_G(A_3,B_3)|\le K|B_3|$, and hence
  \begin{equation*}
    \frac{1}{|A_3|}\sum_{v\in A_3}p_v
    \le \frac{4}{a}\cdot
    \frac{|E_G(A_3,B_3)|+a}{\frac{5}{4}\alpha\ln{n}}
    \le \frac{5K|B_3|}{a\alpha\ln{n}}
    + \frac{5}{\alpha\ln{n}},
  \end{equation*}
  and by the inequality between the arithmetic and geometric means,
  \begin{equation*}
    \prod_{v\in A_3}p_v \le
    \left(\frac{5K|B_3|}{a\alpha\ln{n}} + \frac{5}{\alpha\ln{n}}\right)^{|A_3|}.
  \end{equation*}
  In particular, if $N_\Gamma(A)\subseteq B$ for $B$ with $|B|\le a\beta\ln{n}$
  then there exists $B_3\subseteq B_1$ with $|B_3|\le a\beta\ln{n}$ such that
  $N_\Gamma(A_3)\cap B_2 \subseteq A\cup B_3$, and conditioning on
  $|A_3|\ge a/4$ and making sure $\beta=\beta(K)$ is small enough, the
  probability of that event is at most
  \begin{equation*}
    \prod_{v\in A_3}(p_v)^{\rho\ln{n}/2} \le
    \left(
      \frac{5K\beta}{\alpha} + \frac{5}{\alpha\ln{n}}
    \right)^{a\rho\ln{n}/8}
    \le \left(\frac{6K\beta}{\alpha}\right)^{a\rho\ln{n}/8}.
  \end{equation*}
  Recalling \ref{P:regular}, we notice that
  $\left|N_G(A)\right|\le \frac{4}{3}a\alpha\ln{n}$.  Thus, taking the union
  bound,
  \begin{align*}
    &\pr{\exists A,\ |A|=a:\ \left|N_\Gamma(A)\right|\le a\beta\ln{n}}\\
    &\le \sum_{\substack{A\\|A|=a}}
      \pr{\left|N_\Gamma(A)\right|\le a\beta\ln{n}}\\
    &\le \sum_{\substack{A\\|A|=a}}\left[
      \pr{|A_2|\ge \frac{a}{4}}
      + \pr{\exists B,|B|=a\beta\ln{n},N_\Gamma(A)\subseteq B \mid
          \left|A_3\right|\ge \frac{a}{4}}
    \right]\\
    &\le \binom{n}{a}\left(\frac{16}{\ln{K}}\right)^{a\rho\ln{n}/8}
      + \sum_{\substack{A\\|A|=a}}
        \sum_{\substack{B\subseteq N_G(A)\\|B|=a\beta\ln{n}}}
      \pr{N_\Gamma(A)\subseteq B\mid |A_3|\ge \frac{a}{4}}\\
    &\le o\left(n^{-1}\right)
      + \binom{n}{a}\binom{\frac{4}{3}a\alpha\ln{n}}{a\beta\ln{n}}
        \left(\frac{6K\beta}{\alpha}\right)^{a\rho\ln{n}/8}\\
    &\le o\left(n^{-1}\right)
      + \left[
        \left(\frac{en}{a}\right)^{1/\ln{n}}
        \left(\frac{4\alpha}{\beta}\right)^\beta
        \left(\frac{6K\beta}{\alpha}\right)^{\rho/8}
      \right]^{a\ln{n}},
  \end{align*}
  and we may take $\beta>0$ to be small enough so that the last expression
  will tend to $0$ faster than $1/n$.  Taking the union bound over all
  cardinalities $1\le a\le n/\ln{n}$ implies that the claim holds \whp{} in
  both cases.
\end{proof}

\subsection{Edges between large sets in the trace graph}
\begin{theorem}\label{trace:expansion:2}
  With high probability, there is an edge of $\Gamma$ between every pair of
  disjoint subsets $A,B\subseteq[n]$ satisfying $|A|,|B|\ge
  \frac{n\left(\ln{\ln{n}}\right)^{1.5}}{\ln{n}}$.
\end{theorem}

\begin{proof}
  For each vertex $v\in[n]$ and integer $k\ge0$, let $x_v^k\sim
  \unif\left(N_G(v)\right)$, independently of each other. Let $\nu_t(v)$
  be the number of exits from vertex $v$ by the time $t$. Think of the random
  walk $X_t$ as follows:
  \begin{equation*}
    X_{t+1} =
    \unif\left(\left\{X_t,x_{X_t}^{\nu_t\left(X_t\right)}\right\}\right).
  \end{equation*}
  That is, with probability $1/2$, the walk stays, and with probability
  $1/2$ it goes to a uniformly chosen vertex from $N_G(v)$, independently
  from all previous choices.  Let $\Lambda=n(\ln{\ln{n}})^{1.5}/\ln{n}$, and
  fix two disjoint $A,B$ with $|A|=|B|=\Lambda$.  Denote
  \begin{equation*}
    B' = \left\{v\in B\midd\left|E(v,A)\right|\ge \frac{\Lambda \alpha\ln{n}}
    {2n}\right\},
  \end{equation*}
  and recall that according to property \ref{P:bad} of an
  $\alpha$-pseudo-random graph, $|B'|\ge\frac{\Lambda}{2}$.
  Note also that according to \ref{P:regular}, for every $v\in B'$
  \begin{equation*}
    \frac{\left|A\cap N_G(v)\right|}{\left|N_G(v)\right|}
    \ge \frac{\Lambda\alpha\ln{n}}{2nd_G(v)}
    \ge \frac{3\Lambda}{8n}.
  \end{equation*}

  Let $E_{B,A}$ be the event ``the walk has exited each of the vertices in $B$
  at least $\rho\ln{n}$ times, but has not traversed an edge from $B$ to $A$'',
  and for $v\in B'$, $k>0$, denote by $F_{v,k,A}$ the event ``$\forall0\le i<k,\
  x_v^i\notin A$''. Clearly, $E_{B,A}\subseteq \bigcap_{v\in
  B'}F_{v,\rho\ln{n},A}$, and $F_{v,\rho\ln{n},A}$ are mutually independent for
  distinct vertices $v$, hence
  \begin{equation*}
    \pr{E_{B,A}}
    \le \prod_{v\in B'} \pr{F_{v,\rho\ln{n},A}}
    \le \left(1-\frac{3\Lambda}{8n}\right)^{\rho\ln{n} \cdot\Lambda / 2}
  \end{equation*}
  Let $E$ be the event ``there exist two disjoint sets $A,B$ of size $\Lambda$
  such that the walk $X$ has exited each of the vertices of $B$ at least
  $\rho\ln{n}$ times, but has not traversed an edge from $B$ to $A$''.
  We have that
  \begin{align*}
    \pr{E} &\le
      \sum_{\substack{A\subseteq[n]\\|A|=\Lambda}}
      \sum_{\substack{B\subseteq[n]\\|B|=\Lambda}} \pr{E_{B,A}}\\
    &\le\binom{n}{\Lambda}^2
      \left(1-\frac{3\Lambda}{8n}\right)^{\rho\ln{n} \cdot\Lambda / 2}\\
    &\le \left(\frac{en}{\Lambda}\right)^{2\Lambda}
    \left(1-\frac{3\Lambda}{8n}\right)^{\rho\ln{n} \cdot\Lambda / 2}\\
    &\le \exp\left(2\Lambda\ln{(en/\Lambda)}
    - \frac{3\rho\Lambda^2 \ln{n}}{16n}\right)\\
    &\le \exp\left(3\cdot\frac{n}{\ln{n}}(\ln{\ln{n}})^{2.5}
    - \frac{n}{\ln{n}}(\ln{\ln{n}})^{2.9}\right) = o(1).
  \end{align*}
  Finally, let $E'$ be the event ``there exist two disjoint sets $A,B$ of size
  $\Lambda$ such that the walk $X$ has not traversed an edge from $B$ to
  $A$''.  It follows from \cref{visits} that
  \begin{align*}
    \pr{E'} &= \pr{E',\ \forall v\in [n]:\nu(v)\ge\rho\ln{n}}
    + \pr{E',\ \exists v\in [n]:\nu(v)<\rho\ln{n}}\\
    &= \pr{E} + o(1) = o(1),
  \end{align*}
  and this completes the proof.
\end{proof}

\section{Hamiltonicity and vertex connectivity}\label{sec:hamcon}
This short section is devoted to the proof of \cref{thm:gnp}, which is a
simple corollary of the results presented in the previous sections. In
addition to these results, we will use the following Hamiltonicity criterion
by Hefetz et al:

\begin{samepage}
\begin{lemma}[\cite{HKS09}, Theorem 1.1]\label{ham_crit}
Let $12 \leq d \leq e^{\sqrt[3]{\ln n}}$ and let $G$ be a graph on $n$
vertices satisfying properties $(Q1)$, $(Q2)$ below:
\begin{description}
  \item[(Q1)] For every $S \subseteq [n]$,
    if $|S| \leq  \frac{n \ln \ln n \ln d}{d \ln n \ln \ln \ln n}$,
    then $|N(S)| \geq d|S|$;
\item[(Q2)] There is an edge in $G$ between any two disjoint subsets $A,B
  \subseteq [n]$ such that
  $|A|,|B| \geq  \frac{n \ln \ln n \ln d}{4130 \ln n \ln \ln \ln n}$.
\end{description}
Then $G$ is Hamiltonian, for sufficiently large $n$.
\end{lemma}
\end{samepage}

\begin{proof}[Proof of \cref{thm:gnp}]
  Noting that \cref{blackbox} follows from
  \cref{trace:expansion:1,trace:expansion:2}, and setting $d=\ln^{1/2}n$ in the
  above lemma, we see that its conditions are typically met by the trace
  $\Gamma_L$, $L=(1+\varepsilon)n\ln n$, with much room to spare actually.
  Hence $\Gamma_L$ is \whp{} Hamiltonian.

  \Cref{blackbox} also implies that \whp{} $\Gamma_L$ is an 
  $\left(n/\ln{n},\beta\ln{n}\right)$-expander, for some $\beta>0$,
  and in addition, that there is an edge connecting every two disjoint sets with
  cardinality at least $n(\lln{n})^{1.5}/\ln{n}$.
  Set $k=\beta\ln{n}$, and suppose to the contrary that under these
  conditions, $\Gamma_L^v$ is not $k$-connected. Thus, there is a cut
  $S\subseteq[n]$ with $|S|\le k-1$ such that $[n]\smallsetminus S$ can be
  partitioned into two non-empty sets, $A,B$, with no edge connecting them.
  Without loss of generality, assume $|A|\le|B|$.  If $|A|\le n/\ln{n}$ then
  $k\le \beta\ln{n}|A|\le |N(A)| \le |S|<k$, a contradiction.
  If $n/\ln{n}<|A|<\beta n-k+1$ then take $A_0\subseteq A$ with
  $|A_0|=n/\ln{n}$, and then $\beta n\le |N(A_0)|\le |A\cup S|<\beta n$, again
  a contradiction.  Finally, if $|A|\ge \beta n-k+1$ then $A,B$ are both of size
  at least $n(\lln{n})^{1.5}/\ln{n}$, thus there is an edge connecting the two
  sets, again a contradiction.
\end{proof}

\section{Hitting time results for the walk on \texorpdfstring{$K_n$}{Kn}}
\label{sec:hittingtime}

From this point on, a \emph{lazy} random walk on $K_n$ is a walk which starts
at a uniformly chosen vertex, and at any given step, stays at the current
vertex with probability $1/n$.  Of course, this does not change matters much,
and the random walk of the theorem, including its cover time, can be obtained
from the lazy walk by simply ignoring loops.  Considering the lazy version
makes things much more convenient; observe that for any
$t\ge 0$, the modified random walk is equally likely to be located at any of
the vertices of $K_n$ after $t$ steps, regardless of its history.  Hence, for
any $t$, if we look at the trace graphs $\odd{\Gamma_t}$ and $\even{\Gamma_t}$
formed by the edges (including loops) traversed by the lazy walk at its odd,
respectively even, steps, they are mutually independent, and the graphs formed
by them are distributed as $\hat{G}(n,m)$ with $m=\ceil{t/2}$ and
$m=\floor{t/2}$, respectively, where $\hat{G}(n,m)$ is the random (multi)graph
obtained by drawing independently $m$ edges (with replacement) from all
possible \emph{directed} edges (and loops) of the complete graph $K_n$, and
then ignoring the orientations.  Note that whenever $m=o\left(n^2\right)$, the
probability of a given non-loop edge to appear in $\hat{G}(n,m)$ is
$\sim 2m/n^2$.

Let now for $k\ge 1$,
\begin{align}
  t_-^{(k)} &= n(\ln{n}+(k-1)\ln{\ln{n}}-\ln{\ln{\ln{n}}}),\label{t_-}\\
  t_+^{(k)} &= n(\ln{n}+(k-1)\ln{\ln{n}}+\ln{\ln{\ln{n}}}).
\end{align}
We may as well just write $t_-$ or $t_+$, when $k$ is clear or does not
matter.  Recall the definition of $\tau_C^k$ from \eqref{tauCk}.
The following is a standard result on the coupon collector problem:
\begin{theorem}[Proved in \cite{ER_class}]\label{classical}
  For every $k\ge 1$, \whp{}, $t_-^{(k)} < \tau_C^k < t_+^{(k)}$.
\end{theorem}

To ease notations, we shall denote $\Gamma_+ = \Gamma_{t_+^{(k)}}$ and
similarly $\Gamma_- = \Gamma_{t_-^{(k)}}$. We add a superscript $\mathrm{o}$ or
$\mathrm{e}$ to consider the odd, respectively even, steps only.
We denote the edges of the walk by $\{e_i\mid i>0\}$.

We note that the trace of our walk is typically not a graph, but rather a
multigraph. However, that fact does not matter much, as the multiplicity of
the edges of that multigraph is typically well bounded, as the following
lemma shows:

\begin{lemma}\label{maxmul}
  With high probability, the multiplicity of any edge of $\Gamma_+$ is at
  most 4.
\end{lemma}

\begin{proof}
  Suppose the multiplicity of an edge $e$ in $\Gamma_+$ is greater than 4; in
  that case, its multiplicity in $\odd{\Gamma_+}$ or in $\even{\Gamma_+}$ is at
  least $3$. As $\odd{\Gamma_+}\sim \hat{G}(n,\ceil{t_+/2})$, we have that the probability for
  that to happen is $O(t_+^3/n^6)=o(n^{-2})$. Applying the union bound over
  all possible edges gives the desired result for the odd case (and the even
  case is identical).
\end{proof}

\subsection{\texorpdfstring{$k$}{k}-connectivity}
Clearly, if a given vertex has been visited at most $k-1$ times, or has been
visited $k$ times without exiting the last time, its degree in the trace is
below $2k-1$ or $2k$ respectively, hence $\tau_C^k\le\tau_\delta^{2k-1}$ and
$\tau_C^k +1 \le \tau_\delta^{2k}$; furthermore, if some vertex has a
(simple) degree less than $m$, then removing all of its neighbours from
the graph will disconnect it, hence it is not $m$-vertex-connected, thus
$\tau_\delta^{m} \le \tau_\kappa^{m}$.
To prove \cref{kn_conn} it therefore suffices to prove the following two claims:
\begin{claim}\label{kn_conn_1}
  For any constant integer $k\ge 1$, \whp{} $\tau_C^k \ge \tau_\delta^{2k-1}$
  and $\tau_C^k +1 \ge \tau_\delta^{2k}$.
\end{claim}

\begin{claim}\label{kn_conn_2}
  For any constant integer $m\ge 1$, \whp{} $\tau_\delta^m \ge
  \tau_\kappa^m$.
\end{claim}

\subsubsection{The set \texorpdfstring{$\Small$}{SMALL}}
To argue about the relation between the number of visits of a vertex and its
degree, we would wish to limit the number of loops and multiple edges
incident to a vertex.  This can be easily achieved for small degree vertices,
which are the only vertices that may affect the minimum degree anyway. This
gives motivation for the following definition.

Denote $d_0=\floor{\delta_0\ln{n}}$ for $\delta_0=e^{-20}$.
\begin{equation*}
  \Small = \left\{v\in [n]\mid d_{\odd{\Gamma_-}}(v)<d_0\right\}
\end{equation*}
be the set of all small degree vertices of $\odd{\Gamma_-}$.
Note that the exact value of $\delta_0$ is not important.  We will simply need
it to be small enough for the proof of \cref{final}.

\begin{lemma}\label{smalldeg}
  Let $m\sim n\ln{n}/2$, $\hat{G}\sim \hat{G}(n,m)$ and $v\in[n]$.  Then,
  $\pr{d_{\hat{G}}(v)<d_0}\le n^{-0.9}$.
\end{lemma}

\begin{proof}
  Noting that $d_{\hat{G}}(v)$ is distributed binomially with $2m$ trials and
  success probability $1/n$, we invoke \cref{chernoff:cor} with $c=10$ to obtain
  \begin{equation*}
    \pr{d_{\hat{G}}(v)<d_0}
    \le \exp\left(-\ln{n} \left(1-e^{-10}-10\delta_0\right)(1+o(1)) \right)
    \le n^{-0.9}.\qedhere
  \end{equation*}
\end{proof}

An application of Markov's inequality (since $t_-\sim n\ln{n}$) gives the
following:
\begin{corollary}\label{small:small}
  With high probability, $|\Small|\le n^{0.2}$.
\end{corollary}

\begin{lemma}\label{noloops}
  With high probability, no vertex in $\Small$ is incident to a loop or to a
  multiple edge in $\Gamma_+$.
\end{lemma}

\begin{proof}
  Let $L_v^i$ be the event ``$v$ is incident to a loop in $\Gamma_+$ which is
  the $i$'th step of the random walk''.  Note that we consider loops in
  $\Gamma_+$, which need not be in $\odd{\Gamma_+}$.
   Fix a vertex $v$ and assume it is
  incident to a loop in $\Gamma_+$.  Take $i$ such that the $i$'th step of
  $X_t$, $e_i$, is a loop incident to $v$ (that is, $X_{i-1}=X_i=v$). Let
  $\hat{G}$ be the graph obtained from $\odd{\Gamma_+}$ by removing $e_j$ for
  every odd $i-1\le j\le i+1$.  It is clear then that $\hat{G}$ is distributed
  like $\hat{G}(n,m)$ with $t_+/2-2\le m\le t_+/2-1$ (so $m\sim n\ln{n}/2$),
  and it is independent of the event $L_v^i$.
  Noting that $\pr{L_v^i}=n^{-2}$ and using \cref{smalldeg} we conclude that
  \begin{equation*}
    \pr{L_v^i,v\in\Small}
    \le \pr{L_v^i,d_{\hat{G}}(v)<d_0}
    = \pr{L_v^i}\pr{d_{\hat{G}}(v)<d_0}
    \le n^{-2.9},
  \end{equation*}
  and by applying the union bound over all
  vertices and over all potential times for loops at a vertex we obtain the
  following upper bound for the existence of a vertex from $\Small$ which is
  incident to a loop:
  \begin{equation*}
    \pr{\exists v\in[n],\ i\in [t_+]:\ L^i_v,\ v\in\Small}
    \le n\cdot t_+\cdot n^{-2.9} = o(1).
  \end{equation*}
  Using a similar method, we can show that \whp{} there is no vertex in
  $\Small$ which is incident to a multiple edge in $\Gamma_+$, and this
  completes the proof.
\end{proof}

\begin{lemma}\label{uvo}
  With high probability, for every pair of disjoint vertex subsets
  $U,W\subseteq [n]$ of size $|U|=|W|=n/\ln^{1/2}{n}$, $\odd{\Gamma_-}$ has at
  least $0.5n$ edges between $U$ and $W$.
\end{lemma}

\begin{proof}
  We note that $\left|E_{\odd{\Gamma_-}}(U,W)\right|$ is distributed binomially
  with $\ceil{t_-/2}$ trials and success probability
  $p=\frac{n^2}{\ln{n}}\binom{n+1}{2}^{-1}$.
  As $p>1.9/\ln{n}$, using the Chernoff bounds we have that
  \begin{align*}
    \pr{\left|E_{\odd{\Gamma_-}}(U,W)\right| < 0.5n}
    &\le \pr{\bin(\ceil{t_-/2},1.9/\ln{n}) < 0.5n}\\
    &\le \pr{\bin(n\ln{n}/1.9,1.9/\ln{n}) \le n - 0.5n}
    \le e^{-0.1n},
  \end{align*}
  thus by the union bound
  \begin{align*}
    \pr{\exists U,W: \left|E_{\odd{\Gamma_-}}(U,W)\right| < 0.5n}
    &\le \binom{n}{n/\ln^{1/2}{n}}^2 e^{-0.1n}\\
    &\le \left(e^2\ln{n}\right)^{n/\ln^{1/2}{n}}e^{-0.1n}\\
    &\le \exp\left( \frac{n}{\ln^{1/2}{n}}
    \left(2+\lln{n}\right)-0.1n \right)
    = o(1).\qedhere
  \end{align*}
\end{proof}

\subsubsection{Extending the trace}
Define
\begin{equation*}
  \Gamma_* = \odd{\Gamma_-} + \left\{e_i\mid 1\le i\le \tau_C^k+1,
  e_i\cap\Small\ne\varnothing\right\}.
\end{equation*}

\begin{lemma}\label{mindeg}
  With high probability, $\delta(\Gamma_*)\ge 2k$.
\end{lemma}

\begin{proof}
  Let $v$ be a vertex. If $v\notin\Small$ then $d(v)\ge d_0$ hence \whp{}
  $d'(v)\ge (d_0-8)/4\ge 2k$ (according to \cref{maxmul}). On the other
  hand, if $v\in\Small$, and is not the first vertex of the random walk, then
  it was entered and exited at least $k$ times in the first $\tau_C^k+1$
  steps of the random walk.  By the definition of $\Gamma_*$, all of these
  entries and exits are in $E(\Gamma_*)$.  Since \whp{} none of the vertices in
  $\Small$ is incident to loops or multiple edges (according to
  \cref{noloops}), the simple degree of each such vertex is at least $2k$.

  Noting that \whp{} the first vertex of the random walk is not in $\Small$
  (according to \cref{smalldeg}), we obtain the claim.
\end{proof}

We note that by deleting the edge $e_{\tau_{C}^k+1}$ from $\Gamma_*$ its
minimum degree cannot drop by more than one, so \cref{kn_conn_1}
follows from \cref{mindeg}.

\begin{lemma}\label{maxdeg}
  With high probability, $\Delta\left(\Gamma_*\right) \le 6\ln{n}$.
\end{lemma}

\begin{proof}
  Fix a vertex $v$.  Noting that $d_{\odd{\Gamma_-}}(v)$ is binomially
  distributed with mean $t_-/n\sim\ln{n}$, we invoke \cref{chernoff:cor} with
  $c=-1$ to obtain
  \begin{equation*}
    \pr{d_{\odd{\Gamma_-}}(v) > 3\ln{n}}
    \le \exp\left(-\ln{n} \left(1-e+3\right)(1+o(1)) \right)
    = O\left(n^{-1.2}\right).
  \end{equation*}
  Similarly one can derive
  $\pr{d_{\even{\Gamma_+}}(v)>3\ln{n}} = O\left(n^{-1.2}\right)$.  Since
  $d'_{\Gamma_*}(v)\le d_{\odd{\Gamma_-}}(v) + d_{\even{\Gamma_+}}(v)$ we
  have that $\pr{d_{\Gamma_*}'(v) > 6\ln{n}} = O\left(n^{-1.2}\right)$.  The
  union bound over all vertices gives that
  $\pr{\Delta\left(\Gamma_*\right) > 6\ln{n}} = o(1)$, as we have wished to
  show.
\end{proof}

\begin{lemma}\label{short_paths}
  Fix $\ell \ge 1$. With high probability there is no path of length between
  $1$ and $\ell$ in $\Gamma_*$ such that both of its (possibly identical)
  endpoints lie in $\Small$.
\end{lemma}

\begin{proof}
  For a set $T\subseteq[t_+]$ let $r(T)$ be the minimum number of integer
  intervals whose union is the set of elements from $T$.  In symbols,
  \begin{equation*}
    r(T) = \left| \left\{
    1\le i\le t_+\mid i\in T\land i+1\notin T
    \right\} \right|.
  \end{equation*}
  Fix $\ell\ge 1$ and $P=(v_0,\ldots,v_\ell)$, a path of length $\ell$.
  Suppose first that $v_0\ne v_\ell$.  Let $A$ be the event $P\subseteq
  E\left(\Gamma_+\right)$.  For every set
  $T=\{s_1,\ldots,s_\ell\}\subseteq[t_+]$ with $s_1<s_2<\ldots<s_\ell$,
  let $A_T$ be the event ``$\forall j\in[\ell]$,
  $e_{s_j} = \left\{v_{j-1},v_j\right\}$''.  We have that
  \begin{align*}
    \pr{A,\ v_0,v_\ell\in\Small}
    &\le \sum_{T\in\binom{[t_+]}{\ell}} \pr{A_T,\ v_0,v_\ell\in\Small}\\
    &= \sum_{r=1}^\ell
    \sum_{\substack{T\in\binom{[t_+]}{\ell} \\r(T)=r}}
    \pr{A_T,\ v_0,v_\ell\in\Small}.\\
  \end{align*}
  For every set $T\in\binom{[t_+]}{\ell}$, let
  \begin{equation*}
    I_T = \left\{i\in[t_-] \mid
      i\text{ is odd},\ \nexists s\in T:\ |i-s|\le 1\right\},
  \end{equation*}
  and for a vertex $v\in[n]$, let $d_{I_T}(v)$ be the degree of $v$ in the
  graph formed by the edges $\left\{e_i\mid i\in I_T\right\}$.  Let $D_T(v)$ be
  the event ``$d_{I_T}(v)\le d_0$''.  It follows from the definition of $I_T$
  that $D_T(v_0)$ and $D_T(v_\ell)$ are independent of the event $A_T$.
  Moreover, if $v\in\Small$ then $D_T(v)$ (since
  $d_{I_T}(v) \le d_{\odd{\Gamma_-}}(v)$), and
  as there is exactly one edge of $K_n$ connecting $v_0$ with $v_\ell$,
  conditioning on the event $D_T(v_0)$ cannot increase the probability of the
  event $D_T(v_\ell)$ by much:
  \begin{align*}
    \pr{D_T(v_0),D_T(v_\ell)}
    &\le \pr{D_T(v_0),D_T(v_\ell)\mid \left\{v_0,v_\ell\right\}\notin I_T}\\
    &= \pr{D_T(v_0)\mid \left\{v_0,v_\ell\right\}\notin I_T}
        \pr{D_T(v_\ell)\mid \left\{v_0,v_\ell\right\}\notin I_T}\\
    &\le \pr{D_T(v_0)}\pr{D_T(v_\ell)}
    \cdot\frac{1}{(\pr{\left\{v_0,v_\ell\right\}\notin I_T})^2}\\
    &= \pr{D_T(v_0)}\pr{D_T(v_\ell)}(1+o(1))
    \le n^{-1.7},
  \end{align*}
  here we have used \cref{smalldeg}, and the fact that
  $\left|I_T\right|\sim n\ln{n}/2$. Thus, for a fixed $T$,
  \begin{align*}
    \pr{A_T,\ v_0,v_\ell\in\Small}
    &\le \pr{A_T,\ D_T(v_0),\ D_T(v_\ell)}\\
    &= \pr{A_T} \pr{D_T(v_0),D_T(v_\ell)}\\
    &\le \pr{A_T} \cdot n^{-1.7}.
  \end{align*}
  Similarly, if $v_0=v_\ell$ we obtain
  $\pr{A_T,v_0\in\Small}\le\pr{A_T}\cdot n^{-0.9}$.
  Now, given $T$ with $r(T)=r$ ($1\le r\le \ell$), the probability of $A_T$ is
  at most $n^{-(\ell+r)}$.  It may be $0$, in case $T$ is not feasible, and
  otherwise there are exactly $\ell+r$ times where the walk is forced to be
  at a given vertex (the walk has to start each of the $r$ intervals at a
  given vertex, and to walk according to the intervals $\ell$ steps in
  total), and the probability for each such restriction is $1/n$.  The number
  of $T$'s for which $r(T)=r$ is $O\left(\left(t_+\right)^r\right)$ (choose
  $r$ points from $\left[t_+\right]$ to be the starting points of the $r$
  intervals; then for every $j\in[\ell]$ there are at most $r\ell$ options
  for the $j$'th element of $T$).  Noting that the number of paths of length
  $\ell$ is no larger than $n^{\ell+1}$ if $v_0\ne v_\ell$, or $n^\ell$ if
  $v_0=v_\ell$, the union bound gives
  \begin{align*}
    &\pr{\exists P=(v_0,\ldots,v_\ell):\ A,\ v_0,v_\ell\in\Small}\\
    &\le n^{\ell+1}\sum_{r=1}^\ell
    \frac{O\left(\left(t_+\right)^r\right)}{n^{\ell+r}}
    \cdot n^{-1.7}
    + n^\ell\sum_{r=1}^\ell
    \frac{O\left(\left(t_+\right)^r\right)}{n^{\ell+r}}
    \cdot n^{-0.9}\\
    &\le n^{-0.7}
    \sum_{r=1}^\ell O\left(\ln^r{n}\right) = o(1).\qedhere
  \end{align*}
\end{proof}

\begin{lemma}\label{spans_notmany}
  With high probability, every vertex set $U$ with $|U|\le n/\ln^{1/2}{n}$
  spans at most $2|U|\cdot \ln^{3/4}{n}$ edges (counting multiple edges and
  loops) in $\Gamma_*$.
\end{lemma}

\begin{proof}
  Fix $U\subseteq[n]$ with $|U|=u\le n/\ln^{1/2}{n}$.  Let $\odd{e}(U)$ and
  $\even{e}(U)$ be the number of edges (including multiple edges and loops)
  spanned by $U$ in $\odd{\Gamma_+}$ and $\even{\Gamma_+}$ respectively.  Note
  that
  $\odd{e}(U)$ is binomially distributed with $\ceil{t_+/2}$ trials and success
  probability $u^2/n^2$.  Thus, using \cref{chernoff:trivial} we have that
  \begin{equation*}
    \pr{\odd{e}(U)>u\ln^{3/4}{n}} \le
    \left(\frac{et_+u^2}{2n^2u\ln^{3/4}{n}}\right)
    ^{u\ln^{3/4}{n}}
    \le \left(\frac{e\ln^{1/4}{n}u}{n}\right)^{u\ln^{3/4}{n}}.
  \end{equation*}

  The union bound over all choices of $U$ yields
  \begin{align*}
    \pr{\exists U,\ |U|\le\frac{n}{\ln^{1/2}{n}},\ \odd{e}(U)\ge|U|\ln^{3/4}{n}}
    &\le \sum_{u=1}^{n/\ln^{1/2}{n}}\binom{n}{u}
    \left(\frac{e\ln^{1/4}{n}u}{n}\right)^{u\ln^{3/4}{n}}\\
    &\le \sum_{u=1}^{n/\ln^{1/2}{n}}
    \left(\frac{en}{u}\cdot
    \left(\frac{e\ln^{1/4}{n}u}{n}\right)^{\ln^{3/4}{n}}\right)^u.
  \end{align*}

  We now split the sum into two:
  \begin{equation*}
    \sum_{u=1}^{\ln{n}} \left(\frac{en}{u} \left(\frac{e\ln^{1/4}{n}u}
    {n}\right)^{\ln^{3/4}{n}}\right)^u
    \le \ln{n} \cdot en\left(\frac{e\ln^{5/4}{n}}{n}\right) ^{\ln^{3/4}{n}}
    =o(1),
  \end{equation*}
  and
  \begin{align*}
    \sum_{u=\ln{n}}^{n/\ln^{1/2}{n}} \left(\frac{en}{u}
    \left(\frac{e\ln^{1/4}{n}u} {n}\right)^{\ln^{3/4}{n}}\right)^u
    &= \sum_{u=\ln{n}}^{n/\ln^{1/2}{n}}
    \left( e \left(\frac{u}{n}\right)^{\ln^{3/4}{n}-1}
    \left(e\ln^{1/4}{n}\right)^{\ln^{3/4}{n}}\right)^u\\
    &\le n\left( e \left(\frac{1}{\ln^{1/2}{n}}\right)^{\ln^{3/4}{n}-1}
    \left(e\ln^{1/4}{n}\right)^{\ln^{3/4}{n}}\right)^{\ln{n}} = o(1).
  \end{align*}

  As the same bound applies for $\even{e}(U)$, the union bound concludes the claim
  (noting that ${\Gamma_*\subseteq \Gamma_+}$).
\end{proof}

\begin{lemma}\label{cross_notmany}
  With high probability, for every pair of disjoint vertex sets $U,W$ with
  $|U|\le n/\ln^{1/2}{n}$ and $|W|\le |U|\cdot \ln^{1/4}{n}$, it holds that
  $\left|E_{\Gamma_*}(U,W)\right|\le 2|U|\ln^{0.9}{n}$.
\end{lemma}

\begin{proof}
  For $U,W\subseteq[n]$, $|U|\le n/\ln^{1/2}{n}$, $|W|\le |U|\ln^{1/4}{n}$,
  let $\odd{e}(U,W)$ ($\even{e}(U,W)$) be the number of edges in
  $\odd{\Gamma}_+$ (in $\even{\Gamma}_+$) between $U$ and $W$.  For
  $\mathrm{x}\in\left\{\mathrm{o},\mathrm{e}\right\}$, let
  $A^\mathrm{x}(U,W)$ be the event ``$e^\mathrm{x}(U,W)\ge|U|\ln^{0.9}{n}$'',
  and let $A^\mathrm{x}$ be the event
  ``$\exists U,W,\ |U|\le n/\ln^{1/2}{n},\ |W|\le|U|\ln^{1/4}{n},\
  A^\mathrm{x}(U,W)$''.

  Fix $U,W$ with $|U|=u\le n/\ln^{1/2}{n}$ and $|W|=w\le u\ln^{1/4}{n}$.
  Note that $\odd{e}(U,W)$ is binomially distributed with $\ceil{t_+/2}$ trials
  and success probability $2uw/n^2$.  Thus, using
  \cref{chernoff:trivial} we have that
  \begin{equation*}
    \pr{\odd{e}(U,W) > u\ln^{0.9}{n}}
    \le
    \left(\frac{et_+uw}{n^2u\ln^{0.9}{n}}\right)^{u\ln^{0.9}{n}}
    \le
    \left(\frac{ew\ln^{0.1}{n}}{n}\right)^{\ln^{0.9}{n}}.
  \end{equation*}

  The union bound over all choices of $U,W$ yields
  \begin{align*}
    \pr{\odd{A}}
    &\le \sum_{u=1}^{n/\ln^{1/2}{n}} \sum_{w=1}^{u\ln^{1/4}{n}}
    \binom{n}{u}\binom{n}{w}
    \left(\frac{ew\ln^{0.1}{n}}{n}\right)^{\ln^{0.9}{n}}\\
    &\le \sum_{u=1}^{n/\ln^{1/2}{n}} \sum_{w=1}^{u\ln^{1/4}{n}}
    \left(
    \frac{en}{u}\left(\frac{en}{w}\right)^{w/u}
    \left(\frac{ew\ln^{0.1}{n}}{n}\right)^{\ln^{0.9}{n}}
    \right)^u\\
    &\le \sum_{u=1}^{n/\ln^{1/2}{n}} u\ln^{1/4}{n}
    \left(
    \frac{en}{u}\left(\frac{en}{u\ln^{1/4}{n}}\right)^{\ln^{1/4}{n}}
    \left(\frac{eu\ln^{0.35}{n}}{n}\right)^{\ln^{0.9}{n}}
    \right)^u\\
    &\le \sum_{u=1}^{n/\ln^{1/2}{n}} u\ln^{1/4}{n}
    \left(
    e\left(\frac{u}{n}\right)^{\ln^{0.9}{n}-\ln^{1/4}{n}-1}
    \left(e\ln^{0.35}{n}\right)^{\ln^{0.9}{n}}
    \left(e\ln^{-1/4}{n}\right)^{\ln^{1/4}{n}}
    \right)^u.
  \end{align*}

  We now split the sum into two:
  \begin{align*}
    &
    \sum_{u=1}^{\ln{n}} u\ln^{1/4}{n}
    \left(
    e\left(\frac{u}{n}\right)^{\ln^{0.9}{n}-\ln^{1/4}{n}-1}
    \left(e\ln^{0.35}{n}\right)^{\ln^{0.9}{n}}
    \left(e\ln^{-1/4}{n}\right)^{\ln^{1/4}{n}}
    \right)^u\\
    &\le
    \ln^{9/4}{n}
    \cdot e \left(\frac{\ln{n}}{n}\right)^{\ln^{0.9}{n}-\ln^{1/4}{n}-1}
    \left(e\ln^{0.35}{n}\right)^{\ln^{0.9}{n}}
    \left(e\ln^{-1/4}{n}\right)^{\ln^{1/4}{n}} = o(1),
  \end{align*}
  and
  \begin{align*}
    &
    \sum_{u=\ln{n}}^{n/\ln^{1/2}{n}} u\ln^{1/4}{n}
    \left(
    e\left(\frac{u}{n}\right)^{\ln^{0.9}{n}-\ln^{1/4}{n}-1}
    \left(e\ln^{0.35}{n}\right)^{\ln^{0.9}{n}}
    \left(e\ln^{-1/4}{n}\right)^{\ln^{1/4}{n}}
    \right)^u\\
    &\le
    n^2 \left(
    e\left(\frac{1}{\ln^{1/2}{n}}\right)^{\ln^{0.9}{n}-\ln^{1/4}{n}-1}
    \left(e\ln^{0.35}{n}\right)^{\ln^{0.9}{n}}
    \left(e\ln^{-1/4}{n}\right)^{\ln^{1/4}{n}}
    \right)^{\ln{n}} = o(1).
  \end{align*}

  As the same bound applies for $\mathrm{x}=\mathrm{e}$, the union bound over
  $\mathrm{x}\in\left\{\mathrm{o},\mathrm{e}\right\}$ concludes the claim
  (noting that $\Gamma_*\subseteq \Gamma_+$).
\end{proof}

We will need the following lemma, according to which not too many edges were
added by extending the trace, when we will prove the Hamiltonicity of the
trace:
\begin{lemma}\label{add_notmany}
  With high probability, $\left|E\left(\Gamma_*\right)\smallsetminus
  E\left(\odd{\Gamma_-}\right)\right| \le n^{0.4}$.
\end{lemma}

\begin{proof}
  Recall from \cref{small:small} that \whp{}
  $\left|\Small\right|\le n^{0.2}$. From \cref{maxdeg} it follows that \whp{}
  $\Delta\left(\Gamma_*\right)\le 6\ln{n}$. From \cref{maxmul} it follows
  that \whp{} $d_{\Gamma_*}(v)\le 4\Delta\left(\Gamma_*\right)\le 24\ln{n}$ for
  every $v\in\Small$.  We conclude that the number of edges in $\Gamma_*$ with
  at least one end in $\Small$ is \whp{} at most $n^{0.2}\cdot24\ln{n} <
  n^{0.4}$, and the claim follows by the definition of $\Gamma_*$.
\end{proof}

\subsubsection{Sparsifying the extension}
We may use the results of
\cref{uvo,mindeg,maxdeg,short_paths,spans_notmany,cross_notmany} to show that
$\Gamma_*$ is a (very) good expander.  This, together with \cref{rc_lemma},
will imply that $\Gamma_*$ is $2k$-connected.  However, in order to later show
that $\Gamma_*$ is Hamiltonian, we wish to show it contains a much
\emph{sparser} expander, which is still good enough to guarantee high
connectivity.

To obtain this, we assume $\Gamma_*$ has the properties guaranteed by these
lemmas, and sparsify $\Gamma_*$ randomly as follows: for each vertex $v$, if
$v\in\Small$, define $E(v)$ to be all edges incident to $v$; otherwise let
$E(v)$ be a uniformly chosen subset of size $d_0$ of all edges incident to
$v$. Let $\Gamma_0$ be the spanning subgraph of $\Gamma_*$ whose edge set is
the union of $E(v)$ over all vertices $v$.

\begin{lemma}\label{pairs}
  With high probability (over the choices of $E(v)$), for every pair of
  disjoint vertex sets $U,W\subseteq [n]$ of size $|U|=|W|=n/\ln^{1/2}{n}$,
  $\Gamma_0$ has at least one edge between $U$ and $W$.
\end{lemma}

\begin{proof}
  Let $U,W\subseteq [n]$ with $|U|=|W|=n/\ln^{1/2}{n}$.  From \cref{uvo}
  it follows that in $\Gamma=\odd{\Gamma_-}$ there are at least $0.5n$ edges
  between $U$ and $W$.  If there is a vertex $v\in U\cap\Small$ with an edge
  into $W$, we are done, so we can assume that there is no such.  Let
  $U'=U\smallsetminus\Small$; thus, $\left|E_\Gamma(U',W)\right|\ge 0.5n$.

  Fix a vertex $u\in U'$. Let $X_u$ be the number of edges between $u$ and
  $W$ in $\Gamma$ that fall into $E(u)$.  $X_u$ is a random variable,
  distributed according to
  $\hypg(d_\Gamma(u),\left|E_\Gamma(u,W)\right|,d_0)$.
  According to \cref{chernoff:hypg}, the probability that $X_u=0$ may
  be bounded from above by
  \begin{equation*}
    \exp\left(-\frac{\left|E_\Gamma(u,W)\right|\cdot d_0}{2d_\Gamma(u)}\right),
  \end{equation*}
  which, according to \cref{maxmul,maxdeg}, may be bounded from above by
  \begin{equation*}
    \exp\left(-\frac{\left|E_\Gamma(u,W)\right|\cdot d_0}{50\ln{n}}\right).
  \end{equation*}
  Hence, the probability that there is no vertex $u\in U$ from which there
  exists an edge to $W$ can be bounded from above by
  \begin{equation*}
    \prod_{u\in U'}
     \exp\left(-\frac{d_0}{50\ln{n}}\cdot \left|E_\Gamma(u,W)\right|\right)
    =\exp\left(-\frac{d_0}{50\ln{n}}\cdot \left|E_\Gamma(U',W)\right|\right)
    =\exp\left(-\Theta(n)\right).
  \end{equation*}
  Union bounding over all choices of $U,W$, we have that the probability that
  there exists such a pair of sets with no edge between them is at most
  \begin{equation*}
    \binom{n}{n/\ln^{1/2}{n}}^2e^{-\Theta(n)} \le
    \exp\left(\frac{n}{\ln^{1/2}{n}}(2+\lln{n})-\Theta(n)\right) = o(1).\qedhere
  \end{equation*}
\end{proof}

\begin{lemma}\label{mindeg0}
  $\delta\left(\Gamma_0\right)\ge 2k$.
\end{lemma}

\begin{proof}
  This follows from \cref{mindeg}, since we have not removed any
  edge incident to a vertex from $\Small$ and since any other vertex is
  incident to at least $d_0$ edges.
\end{proof}

Recall the definition of $(R,c)$-expanders from \cref{sec:RCexpanders}.

\begin{lemma}\label{gamma0}
  With high probability (over the choices of $E(v)$) $\Gamma_0$ is a
  $\left(\frac{n}{2k+2},2k\right)$-expander, with at most $d_0 n$ edges.
\end{lemma}

\begin{proof}
  Since by definition $|E(v)|\le d_0$ for every $v\in[n]$, it follows
  that $|E(\Gamma_0)|\le d_0 n$.
  Let $S\subseteq [n]$ with $|S|\le n/(2k+2)$. Denote $S_1=S\cap\Small$ and
  $S_2=S\smallsetminus\Small$. Consider each of the following cases:

  \begin{description}
    \item[Case $|S_2|\ge n/\ln^{1/2}{n}$:] From \cref{pairs} it
      follows that the set of all non-neighbours of $S_2$ (in $\Gamma_0$) is
      of cardinality less than $n/\ln^{1/2}{n}$.  Thus
      \begin{equation*}
        \left|N_{\Gamma_0}(S)\right|
        \ge n - n/\ln^{1/2}{n} - |S| \ge
        \frac{(2k+1)n}{2k+2} - n/\ln^{1/2}{n}
        \ge \frac{2kn}{2k+2} \ge 2k|S|.
      \end{equation*}

    \item[Case $|S_2| < n/\ln^{1/2}{n}$:] From \cref{mindeg0,short_paths}
      it follows that $\left|N_{\Gamma_0}(S_1)\right|\ge 2k|S_1|$.  From
      \cref{spans_notmany} it follows that $S_2$ spans at most
      $2|S_2|\cdot\ln^{3/4}{n}$ edges in $\Gamma_0$.  Consequently,
      \begin{equation*}
        \left|\partial_{\Gamma_0} S_2\right|
        \ge d_0|S_2| - 2\left|E_{\Gamma_0}(S_2)\right|
        > |S_2|(d_0-4\ln^{3/4}{n}) \ge 3|S_2|\cdot \ln^{0.9}{n},
      \end{equation*}
      hence, by \cref{cross_notmany} it holds that
      $\left|N_{\Gamma_0}(S_2)\right| > |S_2|\cdot \ln^{1/4}{n}$.
      Finally, by \cref{short_paths} we obtain that for each $u\in S_2$,
      $\left|N_{\Gamma_0}(S_1)\cap N_{\Gamma_0}^+(u)\right|\le 1$, hence
      \begin{equation*}
        \left|N_{\Gamma_0}(S_1) \cap N_{\Gamma_0}^+(S_2)
        \right| \le |S_2|,
      \end{equation*}
      and thus
      \begin{equation*}
        \left|N_{\Gamma_0}(S_1) \smallsetminus N_{\Gamma_0}^+(S_2) \right|
        \ge 2k|S_1|-|S_2|.
      \end{equation*}

      Similarly, for each vertex in $S_2$ has at most one neighbour in
      $S_1$, thus
      \begin{equation*}
        \left|N_{\Gamma_0}\left(S_2\right)\smallsetminus S_1\right|
        \ge \left|N_{\Gamma_0}\left(S_2\right)\right|-\left|S_2\right|
        > \left|S_2\right|\cdot\ln^{0.2}{n}.
      \end{equation*}

      To summarize, we have that
      \begin{align*}
        \left|N_{\Gamma_0}(S)\right|
        &= \left|N_{\Gamma_0}(S_1) \smallsetminus N_{\Gamma_0}^+(S_2) \right|
        + \left|N_{\Gamma_0}(S_2) \smallsetminus S_1 \right|\\
        &\ge 2k|S_1|-|S_2| + |S_2|\cdot \ln^{0.2}{n}\\
        &\ge 2k\left(|S_1|+|S_2|\right) = 2k|S|.\qedhere
      \end{align*}
  \end{description}

\end{proof}

Since $\Gamma_0$ is \whp{} an $(R,c)$-expander (with
$R(c+1)=\frac{n(2k+1)}{2k+2}\ge\frac{n}{2}+k$), we have that $\Gamma_*$ is
such, and from \cref{rc_lemma} we conclude it is
$2k$-vertex-connected. \Cref{kn_conn_2} follows for even values of $m$.

We have already shown (in \cref{kn_conn_1}) that
$\tau_\delta^{2k-1}+1=\tau_C^k+1=\tau_\delta^{2k}$. Hence, using what we have
just shown we have that $\tau_\delta^{2k-1}+1=\tau_\kappa^{2k}$. Since
removing an edge may decrease connectivity by not more than 1, it follows
that $\tau_\delta^{2k-1}\ge \tau_\kappa^{2k-1}$.

That concludes the proof of \cref{kn_conn_2} and of \cref{kn_conn}.

\subsection{Hamiltonicity}
We start by describing the tools needed for our proof.
\begin{definition}\label{booster}
  Given a graph $G$, a non-edge $e=\left\{u,v\right\}$ of $G$ is called a
  \emph{booster} if adding $e$ to $G$ creates a graph $\tilde{G}$, which is
  either Hamiltonian or whose maximum path is longer than that of $G$.
\end{definition}

Note that technically every non-edge of a Hamiltonian graph $G$ is a booster
by definition.
Adding a booster advances a non-Hamiltonian graphs towards Hamiltonicity.
Sequentially adding $n$ boosters makes any graph with $n$ vertices Hamiltonian.

\begin{lemma}\label{posa}
  Let $G$ be a connected non-Hamiltonian $(R,2)$-expander. Then $G$ has at
  least $(R+1)^2/2$ boosters.
\end{lemma}

The above is a fairly standard tool in Hamiltonicity arguments for random
graphs, based on the P\'osa rotation-extension technique~\cite{P76}. Its proof
can be found, e.g., in~\cite{B01}*{Chapter 8.2}.

We have proved in \cref{gamma0}, for $k=1$, that $\Gamma_*$ (and thus
$\Gamma_{\tau_C+1}$) typically contains a sparse
$\left(\frac{n}{4},2\right)$-expander $\Gamma_0$. We can obviously assume that
$\Gamma_0$ does not contain loops or multiple edges. Expanders are not
necessarily Hamiltonian themselves, but they are extremely helpful in
reaching Hamiltonicity as there are many boosters relative to them by
\cref{posa}. We will thus start with $\Gamma_0$ and will repeatedly add
boosters to it to bring it to Hamiltonicity.  Note that those boosters should
come from within the edges of the trace $\Gamma_{\tau_C+1}$. This is taken care
of by the following lemma.

\begin{lemma}\label{final}
  With high probability, every non-Hamiltonian
    $\left(\frac{n}{4},2\right)$-expander $H\subseteq\Gamma_*$ with
    $\left|E(H)\right|\le d_0n+n$ and
    $\left|E(H)\smallsetminus E(\odd{\Gamma_-})\right|\le n^{0.4}$
    has a booster in $\odd{\Gamma_-}$.
\end{lemma}

\begin{proof}
  For a non-Hamiltonian $\left(\frac{n}{4},2\right)$-expander $H$ let
  $H_{\mathrm{o}}=H\cap \odd{\Gamma_-}$ and
  $H_{\mathrm{e}}=H\smallsetminus H_{\mathrm{o}}$
  be two subgraphs of $H$.  Denote by $\mathcal{B}(H)$ the set of
  boosters with respect to $H$.  At the first stage we choose $H$. For
  that, we first choose how many edges $H$ has (at most $d_0n+n$) and call
  that quantity $i$, then we choose the edges from $K_n$.  At the second stage we choose $H_{\mathrm{e}}$.  For that, we
  first
  choose how many of $H$'s edges are not in $\odd{\Gamma_-}$ (at most $n^{0.4}$)
  and call that quantity $j$, then we choose the edges from $H$.  At the third stage, we require all of $H_{\mathrm{o}}$'s edges
  to
  appear in $\odd{\Gamma_-}$.  For that, we first choose for each edge of
  $H_{\mathrm{o}}$
  a time in which it was traversed, then we actually require that edge to be
  traversed at that time.

  Finally, we wish to bound the probability that given all of the above
  choices, $\odd{\Gamma_-}$ does not contain a booster with respect to $H$. For
  that, recall the definition of $t_-=t_-^{(1)}$ from \eqref{t_-}, let $T$ be the set of
  times in which edges from $H$ were traversed, and as in the proof of
  \cref{short_paths}, define
  \begin{equation*}
    I_T = \left\{i\in[t_-] \mid
      i\text{ is odd},\ \nexists s\in T:\ |i-s|\le 1\right\}.
  \end{equation*}
  Note that in view of \cref{maxmul}, \whp{}
  \begin{equation*}
    \left|I_T\right| \ge \frac{t_-}{2} - 4|E(H)| \ge
    \frac{t_-}{2}-5d_0n \ge \frac{t_-}{3},
  \end{equation*}
  since $\delta_0<1/1000$,
  and observe that every edge traversed in $\odd{\Gamma_-}$ at one of the times
  in $I_T$ is chosen uniformly at random, and independently of all previous
  choices, from all $n^2$ possible directed edges (including loops), so the
  probability of hitting a given (undirected) edge is $2n^{-2}$.  Since
  $H$ is a $\left(\frac{n}{4},2\right)$-expander, it is connected, hence by
  \cref{posa}, $\left|\mathcal{B}(H)\right|\ge n^2/32$, and it follows
  that for $t\in I_T$,
  \begin{equation*}
    \pr{e_t\in\mathcal{B}(H)}\ge \frac{n^2}{32}\cdot 2n^{-2}
    = \frac{1}{16},
  \end{equation*}
  hence for every $H$
  \begin{equation*}
    \prod_{t\in I_T}\pr{e_t\notin \mathcal{B}(H)}
    \le \left(\frac{15}{16}\right)^{t_-/3}.
  \end{equation*}
  To summarize,
  \begin{align*}
    &
    \pr{\exists H:\ \mathcal{B}(H)\cap
    E\left(\odd{\Gamma_-}\right) = \varnothing}
    \\
    &\le
    \sum_{i\le d_0n+n} \binom{\binom{n}{2}}{i}
    \sum_{j\le n^{0.4}}\binom{i}{j}
    \ceil{\frac{t_-}{2}}^{i-j} \left(\frac{2}{n^2}\right)^{i-j}
    \left(\frac{15}{16}\right)^{t_-/3}\\
    &\le
    \left(\frac{15}{16}\right)^{t_-/3}
    \sum_{i\le 2d_0n}\left(\frac{en^2}{2i}\right)^i
    \left[
    n^{0.4}\left(2d_0n\right)^{n^{0.4}}
    \left(\frac{n^2}{t_-}\right)^{n^{0.4}}\right]
    \left(\frac{2t_-}{n^2}\right)^i\\
    &\le
    \left(\frac{15}{16}\right)^{t_-/3}
    n^{\sqrt{n}}
    \sum_{i\le 2d_0n}
    \left(\frac{et_-}{i}\right)^i
    \\
    &\le
    \left(\frac{15}{16}\right)^{t_-/4}
    \sum_{i\le 2d_0n}\left(\frac{et_-}{i}\right)^i.
  \end{align*}
  Let $f(x) = (et_-/x)^x$.  In the interval $(0,et_-)$, $f$ attains its maximum
  at $t_-$, and is unimodal.  Recalling that $d_0=\floor{\delta_0\ln{n}}$
  and that we chose $\delta_0<1/1000$, $f$ is strictly increasing in the
  interval $(0,2d_0n)$.  Thus
  \begin{align*}
    \pr{\exists H:\ \mathcal{B}(H)\cap
    E\left(\odd{\Gamma_-}\right) = \varnothing}
    &\le \left(\frac{15}{16}\right)^{t_-/4}
    2d_0n\left(\frac{et_-}{2d_0n}\right)^{2d_0n}\\
    &\le
    \exp\left(\frac{t_-}{4}\ln\left(\frac{15}{16}\right)
      + 2d_0n\ln\left(\frac{e}{\delta_0}\right)\right)\\
    &\le \exp\left(n\ln{n}\left(
    \frac{\ln(15/16)}{4}
    + 2\delta_0\ln\left(\frac{e}{\delta_0}\right)\right)\right),\\
    &\le \exp\left(n\ln{n}\left(
    -\frac{1}{64} + 42e^{-20}\right)\right) = o(1).\qedhere
  \end{align*}
\end{proof}


Now all ingredients are in place for our final argument. We first state that
\whp{} the graph $\Gamma_*$ contains a sparse
$\left(\frac{n}{4},2\right)$-expander $\Gamma_0$, as delivered by
\cref{gamma0}. We set $H_0=\Gamma_0$, and as long as $H_i$ is not
Hamiltonian, we seek a booster from $\odd{\Gamma_-}$ relative to it; once
such a booster $b$ is found, we add it to the graph and set $H_{i+1}=H_i+b$.
This iteration is repeated less than $n$ times. It cannot get stuck as
otherwise we would get graph $H_i$ for which the following hold:
\begin{itemize}
  \item $H_i$ is a non-Hamiltonian $\left(\frac{n}{4},2\right)$-expander (as
    $H_0\subseteq H_i)$
  \item $\left|E\left(H_i\right)\right|\le d_0 n + n$ (as
    $\left|E\left(\Gamma_0\right)\right|\le d_0 n$)
  \item $\left|E\left(H_i\right)\smallsetminus
    E\left(\odd{\Gamma_-}\right)\right| \le n^{0.4}$ (follows from
    \cref{add_notmany})
  \item $\odd{\Gamma_-}$ does not contain a booster with respect to $H_i$
\end{itemize}
and by \cref{final}, with high probability, such $H_i$ does not exist.

This shows that $\Gamma_{\tau_C+1}$ is \whp{} Hamiltonian; since
$\delta\left(\Gamma_{\tau_C}\right)=1$, $\tau_\mathcal{H}=\tau_C+1$, and the
proof of \cref{kn_ham} is complete.

\subsection{Perfect Matching}
Assume $n$ is even.  Since $\delta\left(\Gamma_{\tau_C-1}\right)=0$, in order
to prove \cref{kn_pm} it suffices to show that $\tau_\mathcal{PM}\le
\tau_C$. Indeed, our proof above shows that \whp{} $\Gamma_{\tau_C}$ contains
a Hamilton path.  Taking every second edge of that path, including the last
edge, yields a matching of size $n/2$, thus \whp{} $\Gamma_{\tau_C}$ contains
a matching of that size, and \cref{kn_pm} follows.

\section{Concluding remarks}\label{sec:concluding}
We have investigated several interesting graph properties (minimum degree,
vertex-connectivity, Hamiltonicity) of the trace of a long-enough random walk
on a dense-enough random graph, showing that in the relevant regimes, the
trace behaves much like a random graph with a similar density.  In the
special case of a complete graph, we have shown a hitting time result, which
is similar to standard results about random graph processes.

However, the two models are, in some aspects, very different.  For example,
an elementary result from random graphs states that the threshold for the
appearance of a vertex of degree $2$ is $n^{-3/2}$, whereas the expected
density of the trace of the walk on $K_n$, at the moment the maximum degree
reaches $2$, is of order $n^{-2}$ (as it typically happens after two steps).
It is therefore natural to ask for which graph properties and in which regimes
the two models are alike.

Further natural questions inspired by our results include asking for the
properties of the trace of the walk in different random environments, such as
random regular graphs, or in deterministic environments, such as
$(n,d,\lambda)$-graphs and other pseudo-random graphs (see~\cite{KS06} for a
survey).  We have decided to leave these questions for a future research.

A different direction would be to study the directed trace.  Consider the set
of \emph{directed} edges traversed by the random walk. This induces a random
directed (multi)graph, and we may ask, for example: is it true that when
walking on the complete graph, typically one step after covering the graph we
achieve a \emph{directed} Hamilton cycle?

\begin{acknowledgement}
  The authors wish to thank Wojciech Samotij and Yinon Spinka for useful
  discussions, and the anonymous referee for careful reading and valuable
  comments and suggestions.
\end{acknowledgement}

\newpage

\bibliography{library}
\end{document}